\documentclass[oneside]{amsart}

\usepackage{amsmath,amssymb,mathtools}
\usepackage{tikz,enumitem,float,graphicx,algorithm,algpseudocode,caption}
\usepackage[warningthreshold=1]{resizegather}

\makeatletter
\def\paragraph{\@startsection{paragraph}{4}%
  \z@\z@{-\fontdimen2\font}%
  {\normalfont\bfseries}}
\makeatother

\usepackage{mleftright}\mleftright
\usepackage[T1]{fontenc}
\usepackage{bm}
\usepackage[scr=boondox]{mathalfa}
\newcommand{\Supp}[1]{\operatorname{Supp}\left(#1\right)}
\newcommand{\Newt}[1]{\operatorname{\mathscr{N\kern-3pt e\kern-1pt w\kern-1pt t}}\left(#1\right)}
\DeclareRobustCommand{\stirling}{\genfrac\{\}{0pt}{}}
\newcommand{\zetax}{\zeta_{2(p-1)}}
\newcommand{\calE}{\mathcal{E}}
\newcommand{\calO}{\mathcal{O}}
\newcommand{\calR}{\mathcal{R}}
\newcommand{\bbC}{\mathbb{C}}
\newcommand{\bbF}{\mathbb{F}}
\newcommand{\bbL}{\mathbb{L}}
\newcommand{\bbN}{\mathbb{N}}
\newcommand{\bbQ}{\mathbb{Q}}
\newcommand{\bbR}{\mathbb{R}}
\newcommand{\bbZ}{\mathbb{Z}}
\newcommand{\frakA}{\mathfrak{A}}
\newcommand{\frakm}{\mathfrak{m}}
\newcommand{\frakq}{\mathfrak{q}}
\newcommand{\fraks}{\mathfrak{s}}
\newcommand{\frakz}{\mathfrak{z}}

\usepackage{geometry}

\usepackage[pdfencoding=auto,psdextra]{hyperref}
\usepackage[nameinlink]{cleveref}

\usepackage{amsthm}
\newtheorem{theorem}{Theorem}[section]
\newtheorem*{theorem*}{Theorem}
\newtheorem{lemma}[theorem]{Lemma}
\newtheorem{remark}[theorem]{Remark}
\newtheorem{proposition}[theorem]{Proposition}
\newtheorem{definition}[theorem]{Definition}
\newtheorem{corollary}[theorem]{Corollary}
\newtheorem{example}[theorem]{Example}
\Crefname{example}{Example}{Examples}
\Crefname{lemma}{Lemma}{Lemmas}
\Crefname{definition}{Definition}{Definitions}
\Crefname{remark}{Remark}{Remarks}
\Crefname{proposition}{Proposition}{Propositions}
\renewcommand{\theequation}{\alph{equation}}
\Crefname{equation}{}{}
\numberwithin{equation}{section}
\numberwithin{figure}{section}
\Crefname{item}{}{}
\algnewcommand\algorithmicinput{\textbf{INPUT:}}
\algnewcommand\INPUT{\item[\algorithmicinput]}
\algnewcommand\algorithmicoutput{\textbf{OUTPUT:}}
\algnewcommand\OUTPUT{\item[\algorithmicoutput]}
\Crefname{algorithm}{Algorithm}{Algorithms}

\allowdisplaybreaks

\title{Uniformizer of the False Tate Curve Extension of $\bbQ_p$}
\author{Shanwen Wang}
\address{School of Mathematics, Renmin University of China, Beijing, China}
\email{s\_wang@ruc.edu.cn}
\author{Yijun Yuan}
\address{School of Mathematical Sciences, Fudan University, Shanghai, China}
\email{941201yuan@gmail.com}
\thanks{The first author is supported by the Fundamental Research Funds for the Central Universities,  the Research Funds of Renmin University of China \textnumero 20XNLG04 and The National Natural Science Foundation of China (Grant \textnumero 11971035).}

\begin{document}
\begin{abstract}
    Let $p\geq 3$ be a prime number. In this article, we study the canonical expansion of the primitive $p^n$-th root of unity $\zeta_{p^n}$ in $p$-adic Mal'cev-Neumann field $\bbL_p$ for $n\geq 1$. More precisely, we give the explicit formula for the first $\aleph_0$ terms of the expansion of $\zeta_{p^n}$ and as an application, we use it to construct a uniformizer of $K_{2,m}=\bbQ_p\left(\zeta_{p^2},p^{1/p^m}\right)$ with $m\geq 1$.
\end{abstract}
\maketitle

\setcounter{tocdepth}{2}
\tableofcontents

\section{Introduction}
\subsection{Motivation}
Let $p\geq 3$ be a prime number. For an integer $n\geq 1$, let $\mu_{p^n}$ be the group of $p^n$-th roots of unity and we fix a compatible system $\epsilon=(\zeta_{p^n}\in \mu_{p^n})_{n\geq 0}$ of primitive $p^n$-th root of unity (i.e., for any $l\leq n$, we have $\zeta_{p^n}^{p^l}=\zeta_{p^{n-l}}$).
For $n\geq m\geq 0$ two integers, we denote by $K_{n,m}=\bbQ_p\left(\mu_{p^n},p^{1/p^m}\right)$ the false Tate curve extension of $\bbQ_p$, which is a finite Galois extension of $\bbQ_p$ of degree $\varphi(p^n)p^m$.  Let $\Gamma$ be the Galois group of $\bbQ_p^{\mathrm{cycl}}=\cup_n K_{n,0}$ over $\bbQ_p$ and let $\Gamma^{\mathrm{FT}}$ be the Galois group of $K_{\infty}=\cup_{n} K_{n,n}$ over $\bbQ_p$. Both of them are $p$-adic Lie groups.

Let $\calO_{\calE}^{+}=\mathbb{Z}_p[[T]]$, $\calO_{\calE}$ the $p$-adic completion of $\calO_{\calE}^{+}[\frac{1}{T}]$ and $\calE=\calO_{\calE}^{+}[\frac{1}{p}]$ the fraction field of $\calO_{\calE}$. The field $\calE$ is equipped with actions of $\varphi$ and $\Gamma$ given by the formulae:
\[\varphi(T)=(1+T)^p-1, \gamma(T)=(1+T)^{\chi_{\mathrm{cycl}}(\gamma)},\]
where $\chi_{\mathrm{cycl}}$ is the cyclotomic character. An \'etale $(\varphi,\Gamma)$-module over $\calE$ is a finite dimensional $\calE$-vector space $D$ endowed with semi-linear actions of $\varphi$ and $\Gamma$ commuting with each other, such that $\varphi^{*}D\cong D$. If $D$ is an \'etale $(\varphi,\Gamma)$-module over $\calE$, all the elements $x$ of $D$ can be uniquely written in the form $x=\sum_{i=0}^{p-1}\varphi(x_i)(1+T)^i$, with $x_i\in\calE$. This allows us to define a left inverse $\psi: D\rightarrow D$ of $\varphi$ by the formula $\psi(x)=x_0$; moreover, $\psi$ commutes with $\Gamma$.

The theory of $(\varphi, \Gamma)$-modules introduced by Fontaine is fruitful for the study of commutative Iwasawa theory. Among the other things, the fundamental result is the following theorem (cf. \cite[Th\'{e}or\`{e}me II.1.3]{CherbonnierColmez1999}):
\begin{theorem}[Fontaine]For any $p$-adic representation $V$, there is an isomorphism of $\bbZ_p[[\Gamma]]$-modules
    $$\mathrm{Exp}^{*}:\mathrm{H}^1_{\mathrm{Iw}}(\bbQ_p, V)\cong D(V)^{\psi=1},$$
    where $\mathrm{H}^1_{\mathrm{Iw}}(\bbQ_p, V)= \mathrm{H}^1\left(\mathrm{Gal}(\bar{\bbQ}_p/\bbQ_p ),\bbZ_p[[\Gamma]] \otimes V\right)$ and $D(V)$ is the $(\varphi,\Gamma)$-module associated to $V$ by Fontaine's equivalence of categories (cf. \cite[Th\'{e}or\`{e}me 3.4.3]{Fontaine1990}).
\end{theorem}

In 2004, J. Coates, T. Fukaya, K. Kato, R. Sujatha, and O. Venjakob \cite{CFKSV2005} proposed a program of non-commutative Iwasawa theory. In view of the important role played by the theory of $(\varphi, \Gamma)$-modules in commutative Iwasawa theory, it is natural to ask if there is an analogy of the $(\varphi, \Gamma)$-module theory in the non-commutative situation. The first interesting case can be the tower of the false Tate curve extension of $\bbQ_p$.  In \cite{Ribeiro2011}, Ribeiro introduced the notion of cohomology of $(\varphi, \Gamma^{\rm FT})$-modules. But this definition seems very difficult to describe non-commutative Iwasawa cohomology. A more direct way could be intimating the theory of field of norms of Fontaine and Wintenberger in this case and rebuild the whole theory. One surprising obstruction is that we do not even know how to write down a norm-compatible system of uniformizers of the tower $\{K_{n,m}\}_{n\geq m\geq 0}$ explicitly, which is a key input in the theory of field of norms.

Recently, there are some attempts to attack this problem. In \cite{Viviani2004}, Viviani gave a uniformizer of $K_{1,m}$:
$$ \pi_{1,m}=\frac{1-\zeta_p}{\prod_{i=1}^{m}p^{\frac{1}{p^i}}}.$$
If we denote by $v_{1,m}$ the $p$-adic valuation on $K_{1,m}$ normalized by $v_{1,m}(p)=p^m(p-1)$, then
$v_{1,m}(1-\zeta_p)=p^m$ and $v_{1,m}(p^{\frac{1}{p^m}})=p-1$ which are coprime to each other. Thus, one can use B\'ezout's lemma to construct a uniformizer in this case.
Bellemare and Lei \cite{BellemareLei2020} expand an idea of the user ``Mercio'' on the website Stackexchange and constructed a uniformizer for the field $K_{2,1}$, and they explain the reason why their method can't go further. In this article, we extend an idea of Lampert (cf. \cite{LampertMOF1}) to construct a uniformizer of $K_{2,m}$ with $m\geq 1$.

\subsection{Main results}
\paragraph*{Convention} Let $[\cdot]:\bar{\bbF}_p\rightarrow W(\bar{\bbF}_p)=\calO_{\breve{\bbQ}_p}$ be the Teichm\"{u}ller character, where $W(\bar{\bbF}_p)$ is the ring of Witt vectors over $\bar{\bbF}_p$. For any positive integer $k$ that coprimes to $p$, by abuse of notations, we will not distinguish the symbol of $k$-th primitive root $\zeta_k$ in $\bar{\bbF}_p$ and its Teichm\"{u}ller representative $[\zeta_k]$ in $\calO_{\breve{\bbQ}_p}$.

As we observed in the case $K_{1,m}$, if one can find an algebraic integer of $K_{n,m}$ with valuation coprime to $p$, then we can use B\'ezout's lemma to construct a uniformizer of $K_{n,m}$.

David Lampert in his paper \cite{Lampert1986} gave the $p$-adic expansion of $\zeta_{p^2}$ without a proof\footnote{Lampert claimed at \cite{LampertMOF1,LampertMOF2} that the expansion in his paper is incorrect.}.  The formula appearing in his paper indicates (cf. \cite{LampertMOF1}) that there is a chance to construct the desired algebraic integer. This leads us to study the canonical expansion of the primitive root of unity $\zeta_{p^n}$ in the $p$-adic Mal'cev-Neumann field $\bbL_p$, which is the spherical completion of $\bbC_p$ (cf. \Cref{subsec:8394}). On the other hand, Kedlaya\cite{Kedlaya2001} used a transfinite induction to prove the algebraic closeness of the $p$-adic Mal'cev-Neumann field $\bbL_p$. We expand Kedlaya's proof into a transfinite Newton's algorithm in  \Cref{algorithmtheory}. We will call the result of the $i$-th step of the transfinite Newton algorithm for a given polynomial $P(T)\in \bbL_p[T]$, the $i$-th approximation of a root of $P(T)$. Using this algorithm, we prove an explicit formula for the first $\aleph_0$ terms of canonical expansion of a $p^n$-th primitive root of unity in $\bbL_p$ for every $n\geq 2$ (cf. \Cref{maintheorem1} of local cite):
\begin{theorem*}
    Let $\zeta_{p^n}^{(i)}$ in $\bbL_p$ be the $i$-th approximation of $\zeta_{p^n}$ in the transfinite Newton algorithm for all $n\geq 2$. Then we have
    $$\zeta_{p^n}^{(i)}=
        \begin{cases}
            \sum_{k=0}^i \frac{(-1)^{kn}}{[k!]}\zetax^kp^{\frac{k}{p^{n-1}(p-1)}},                        & \text{ for } 0\leq i\leq p-1, \\
            \zeta_{p^n}^{(p-1)}+ \sum_{l=n}^{i-p+n}(-1)^n\zetax p^{\frac{1}{p^{n-2}(p-1)}-\frac{1}{p^l}}, & \text{ for }  i\geq p.
        \end{cases}
    $$
\end{theorem*}
In other words, we have
$$\zeta_{p^n}=\sum_{i=0}^{p-1}\frac{(-1)^{in}}{[i!]}\zeta_{2(p-1)}^ip^{\frac{i}{p^{n-1}(p-1)}}+\sum_{i=n}^\infty (-1)^n\zeta_{2(p-1)}p^{\frac{1}{p^{n-2}(p-1)}-\frac{1}{p^i}} +O\left(p^{\frac{1}{p^{n-2}(p-1)}}\right) .$$
In the same theorem (cf. \Cref{maintheorem1}), we discuss the relation among all the possibilities of the first $\aleph_0$ terms of $p^n$-th primitive roots:
\begin{theorem*}
    For every $\alpha\in\bbL_p$, denote by $\aleph_0(\alpha)$ the first $\aleph_0$ terms of the expansion of $\alpha$ in $\bbL_p$.
    \begin{enumerate}
        \item There exists $p-1$ distinct elements $m_0,\cdots,m_{p-2}$ in $\left(\bbZ/p^n\bbZ\right)^\times$ such that
              \begin{enumerate}
                  \item $$\aleph_0\left(\zeta_{p^n}^{m_k}\right)=\sum_{i=0}^{p-1}\frac{\left((-1)^n\zetax^{2k+1}\right)^i}{[i!]} p^{\frac{i}{p^{n-1}(p-1)}}+\sum_{j=n}^\infty (-1)^n\zetax^{2k+1} p^{\frac{1}{p^{n-2}(p-1)}-\frac{1}{p^j}};$$
                  \item $\calR_n=\{m_0,\cdots,m_{p-2}\}$ forms a $\mathrm{mod}\ p$ residue system of $\left(\bbZ/p^n\bbZ\right)^\times$.
              \end{enumerate}
        \item For every $m\in\left(\bbZ/p^n\bbZ\right)^\times$, there exists a unique $m_t\in\calR_n$ such that $\aleph_0\left(\zeta_{p^n}^m\right)=\aleph_0\left(\zeta_{p^n}^{m_t}\right)$.
    \end{enumerate}
\end{theorem*}
Besides that, we give an analogous result for $\zeta_p$ in \Cref{prop:1238}.
Finally, using \Cref{maintheorem1}, we construct a uniformizer of $K_{2,m}$ (cf. \Cref{thm:uniformizer} of local cite):
\begin{theorem*}
    \begin{enumerate}
        \item The element
              $$\pi_{2,1}= \left(p^{\frac{1}{p}}\right)^{-1}\left(\zeta_{p^2}-\sum_{k=0}^{p-1}\frac{1}{[k!]}\zetax^kp^{\frac{k}{p(p-1)}}\right)$$
              is a uniformizer of $K_{2,1}$.
        \item For $m\geq 2$, the element
              $$\pi_{2,m}= \left(p^{\frac{1}{p^m}}\right)^{-\frac{p^m-1}{p-1}}\left(\zeta_{p^2}-\sum_{k=0}^{p-1}\frac{1}{[k!]}\zetax^kp^{\frac{k}{p(p-1)}}-\sum_{l=2}^m \zetax p^{\frac{1}{p-1}-\frac{1}{p^l}}\right)$$
              is a uniformizer of $K_{2,m}$.
    \end{enumerate}
\end{theorem*}

\begin{remark}
    If one can give the explicit formula for the second $\aleph_0$ terms of the canonical expansion of the $p^n$-th primitive root of unity in $\bbL_p$, then it is possible that our strategy can go further to find a uniformizer in more general cases.
\end{remark}

\section{Transfinite Newton algorithm}\label{algorithmtheory}
In this paragraph, we summarize the properties of the $p$-adic Mal'cev-Neumann field $\bbL_p$ in \Cref{subsec:8394}, and expand Kedlaya's proof of the algebraic closeness of $\bbL_p$ into a transfinite Newton algorithm in \Cref{subsec:65394}.

\subsection{The $p$-adic Mal'cev-Neumann field $\bbL_p$}\label{subsec:8394}
Let $\calO_{\breve{\bbQ}_p}=W(\bar{\bbF}_p)$ be the ring of Witt vectors over $\bar{\bbF}_p$ and let $\bbL_p$ be the $p$-adic Mal'cev-Neumann field $\calO_{\breve{\bbQ}_p}((p^{\bbQ}))$ (cf. \cite[Section 4]{Poonen1993}). Every element $\alpha$ of $\bbL_p$ can be uniquely written as
\begin{equation}\sum_{x\in \bbQ}[\alpha_x]p^x, \text{ where } [\cdot]: \bar{\bbF}_p\rightarrow W(\bar{\bbF}_p) \text{ is the Teichm\"{u}ller character.}\end{equation}  For any $\alpha=\sum_{x\in \bbQ}[\alpha_x]p^x\in \bbL_p$, we set $ \mathrm{Supp}(\alpha)=\{x\in \bbQ:  \alpha_x\neq 0\}$, which is well-orderd by the definition of $\bbL_p$.
Thus, we can define the $p$-adic valuation $v_p$ by the formula: $$v_p(\alpha)=\begin{cases} \inf \mathrm{Supp}(\alpha), & \text{ if } \alpha\neq 0; \\ \infty, & \text{if } \alpha=0\end{cases} .$$
The field $\bbL_p$ is complete for the $p$-adic topology and it is also algebraically closed. Moreover, it is the maximal complete immediate extension\footnote{A valued field extension $(E,w)$ of  $(F,v)$ is an immediate extension, if  $(E,w)$ and $(F,v)$ have the same residue field. A valued field  $(E,w)$  is maximally complete if it has no immediate extensions other than $(F,v)$ itself.} of $\overline{\bbQ}_p$.
\begin{remark}$\bbL_p$ is spherical complete\footnote{A valued field is said to be spherical complete, if the intersection of every decreasing sequence of closed balls is nonempty.}, and $\bbC_p$ is not spherical complete. The field $\bbC_p$ of $p$-adic complex numbers can be continuously embedded into $\bbL_p$.
\end{remark}
Given $\alpha\in \bbL_p$, for $x\in\bbQ$, we denote the coefficient of $p^x$ in the expansion of $\alpha$ by $[C_x(\alpha)]\in \calO_{ \breve{\bbQ}_p}$. Then $C_x(\alpha)\in\bar{\bbF}_p$ equals to $[C_x(\alpha)]$ modulo $p$. This gives a map
$$ C: \bbQ\times\bbL_p\rightarrow  \bar{\bbF}_p;  (x, \alpha)\mapsto C_x(\alpha) .$$The following lemma summaries the basic properties of the map $C$.
\begin{lemma}\label{lemma_c_g}For every $x,y\in \bbQ$ and $\alpha,\beta\in \bbL_p$, we have
    \begin{enumerate}
        \item if $v_p(\alpha)>x$, then $C_x(\alpha)=0$;
        \item $C_x(p^{-y}\alpha)=C_{x+y}(\alpha)$;
        \item for every $\bar{u}\in\bar{\bbF}_p$ and $u=[\bar{u}]\in\calO_{ \breve{\bbQ}_p} $, we have $\bar{u}C_x(\alpha)=C_x(u\alpha)$;
        \item if $v_p(\alpha), v_p(\beta)\geq x$, then $C_x(\alpha)\pm C_x(\beta)=C_{x}(\alpha\pm \beta)$.
    \end{enumerate}
\end{lemma}
\subsection{Transfinite Newton algorithm}\label{subsec:65394}
\begin{definition}[Newton polygon]\leavevmode Suppose $(K,v)$ is a valued field with value group $\bbQ$. Let $J(T)=\sum_{i=0}^n a_{n-i}T^i\in K[T]$ be a nonzero polynomial. For $0\leq i\leq n$, we have the points $(i, v(a_{i}))\in \bbN\times \bar{\bbR}$, where $\bar{\bbR}= \bbR\cup\{+\infty\}$. If $a_i=0$, $(i,v(a_i))$ is regarded as $Y_{+\infty}$, the point at infinity of the positive vertical axis.
    \begin{enumerate}
        \item Define the \textbf{Newton polygon} $\Newt{J}$ of $J(T)$ as the lower boundary of the convex hull of the points $(i,v(a_i))$ for $i=0,\cdots,n$. As a consequence, $\Newt{J}$ is a function on $\bbR_{\geq 0}$ with values in $\bar{\bbR}$.
        \item The integers $m$ such that $(m,v(a_m))$ are vertices of $\Newt{J}$ are called the \textbf{breakpoints}, and we denote by $m_{\max}^J$ the \textbf{largest breakpoint} less than $n$.
        \item Given two adjacent breakpoints $m^J_1<m^J_2$, denote by $s^J_{m_1}=\frac{v(a_{m^J_2})-v(a_{m^J_1})}{m^J_2-m^J_1}$, the \textbf{slope} of constituent segment of $\Newt{J}$ with endpoints $(m^J_1,v(a_{m^J_1}))$ and $(m^J_2,v(a_{m^J_2}))$. The \textbf{largest slope} is denoted by $s_{\max}^J= s_{m_{\max}^J}^J=\frac{v(a_n)-v(a_{m_{\max}^J})}{n-m_{\max}^J}$. If $(n,v(a_n))=Y_{+\infty}$ (i.e. $a_n=0$)\footnote{Notice that if $m$ is a breakpoint, then $(m,v(a_m))=Y_{+\infty}\Leftrightarrow m=n$ and $a_n=0$.}, we regard $s_{\max}^J=\infty$. Thus, $s_{\max}^\bullet$ is a map from $K[T]$ to $\bbQ\cup\{\infty\}$.
    \end{enumerate}
    We will omit the superscript $J$ if there is no confusion.
\end{definition}
Let $P(T)=a_0T^n+a_1 T^{n-1}+\cdots+ a_n\in \bbL_p[T] $ be a polynomial with $a_n\neq 0$.  For any $u\in\calO_{\bbL_p}^{*} $, set
$$P_u(T)=P(T+ up^{s^P_{\max}}),$$
where $s^P_{\max}$ is the maximal slope of the Newton polygon of $P$.
\begin{lemma}\label{lem:15792}Let $P(T)=a_0T^n+a_1 T^{n-1}+\cdots+ a_n\in \bbL_p[T] $ be a polynomial with $a_n\neq 0$.  For any $u\in\calO_{\bbL_p}^{*} $, we write $P_u(T)=\sum_{i=0}^{n}b_{n-i}T^i$.

    Then one has:
    \begin{enumerate}
        \item The Newton polygons $\Newt{P_u}$ and $\Newt{P}$ are identical in the range $\left[0, m_{\max}^P\right]$;
        \item If $m_{\max}^{P}<k\leq n$, then the point $(k, v_p(b_k))$ is on or above $\Newt{P}$, in other words, we have
              $$v_p(b_k)\geq v_p(a_{m_{\max}^P})+s_{\max}^P(k-m_{\max}^P).$$

    \end{enumerate}
\end{lemma}
\begin{proof}For simplification of notations, we set $s=s_{\max}^P$ and $m=m_{\max}^{P}$. We calculate the $p$-adic valuation of
    $$b_k=\sum_{j=0}^{k}a_{k-j}\binom{n-k+j}{j}u^j p^{sj} \textbf{.}$$
    Note that $v_p\left(a_{k-j}\binom{n-k+j}{j}u^j p^{sj}\right)=v_p(a_{k-j})+sj$.
    \begin{enumerate}[wide, labelindent=0pt]
        \item \par Suppose $k\leq m$ is a breakpoint of $\Newt{P}$. If $j>0$, one observes
              $$v_p(a_{k-j})+sj=v_p(a_k)+j\left(s-\frac{v_p(a_k)-v_p(a_{k-j})}{k-(k-j)}\right).$$
              Since $s$ is the maximal slope of $\Newt{P}$ and $i$ is a breakpoint, one has $\frac{v_p(a_k)-v_p(a_{k-j})}{k-(k-j)}<s$. In other words, for all $j>0$, we have $v_p\left(a_{k-j}\binom{n-k+j}{j}u^j p^{sj}\right)>v_p(a_k)$. As a consequence, in this case, we have $v_p(b_k)=v_p(a_k)$.
              \par Now suppose that $k< m$ is not a breakpoint of $\Newt{P}$. Let $m^P_1<m^P_2$ be two adjacent breakpoints of $P$ such that $m^P_1<k<m^P_2$. We claim that: for all $0\leq j\leq k$, we have\begin{equation}\label{claim}v_p(a_{k-j})+sj\geq (k-m^P_1)s_{m^P_1}+v_p(a_{m^P_1}).\end{equation}
              This claim implies that
              $$v_p(b_k)\geq (k-m^P_1)s_{m^P_1}+v_p(a_{m^P_1}) ,$$
              i.e. the point $(k,v_p(b_k))$ is on or above $\Newt{P}$.
              \par In the following, we prove the claim \Cref{claim}. Since $s_{m^P_1}<s$, one has
              $$sj-(k-m^P_1)s_{m^P_1}\geq s_{m^P_1}(m^P_1-(k-j)).$$
              \begin{enumerate}
                  \item If  $k-j=m^P_1$, we have $$v_p(a_{k-j})+sj=v_p(a_{m^P_1})+sj\geq v_p(a_{m^P_1}) +s_{m^P_1}j=v_p(a_{m^P_1}) +s_{m^P_1}(k-m^P_1) .$$
                  \item If $k-j<m^P_1$, we have
                        $$\frac{sj-(k-m^P_1)s_{m^P_1}}{m^P_1-(k-j)}\geq s_{m^P_1}\geq \frac{v_p(a_{m^P_1})-v_p(a_{k-j})}{m^P_1-(k-j)}.$$
                  \item If $k-j>m_1^P$, one has
                        $$\frac{v_p(a_{k-j})-v_p(a_{m^P_1})}{(k-j)-m^P_1}\geq s_{m^P_1}\geq \frac{s_{m^P_1}(k-m^P_1)-sj}{(k-j)-m^P_1} .$$
              \end{enumerate}
        \item  The second assertion follows from the same discussion.
    \end{enumerate}
\end{proof}
\begin{definition}
    For any polynomial $P(T)=a_0T^n+a_1 T^{n-1}+\cdots+ a_n\in \bbL_p[T] $, we set $s=s_{\max}^P$ and $m=m_{\max}^{P}$. We define a polynomial
    $$\mathrm{Res}_P(T)=\sum_{k=0}^{n-m}C_0\left(a_{n-k}p^{-v_p(a_{m})-s(n-m-k)}\right)T^k \in \bar{\bbF}_p[T],$$ called the \textit{residue polynomial} associated to $P(T)$.
\end{definition}
\begin{remark}
    From the geometric point of view, the residue polynomial can be constructed as follows:
    \begin{enumerate}
        \item Intercept the segment with maximal slope.
        \item Record those coefficients $a_i$ of $P(T)$ with $(i, v_p(a_i))$ lying on this segment as $b_i$. The other coefficients $b_i$ should be recorded as $0$.
        \item Using the coefficients $(b_i)$ in the previous step and the map $\bbL_p\rightarrow \bar{\bbF}_p, \alpha\mapsto C_{v_p(\alpha)}(\alpha)$, one can construct the polynomial $$\mathrm{Res}_P(T)=\sum_{k=0}^{n-m}C_{v_p(a_{n-k})}\left(a_{n-k}\right)T^k\in \bar{\bbF}_p[T].$$  Note that $v_p(a_{n-k})=v_p(a_{m})+s(n-m-k)$ and we have $C_{v_p(\alpha)}(\alpha)=C_0(\alpha\cdot p^{-v_p(\alpha)})$  by \Cref{lemma_c_g}.
    \end{enumerate}
\end{remark}
\begin{proposition}\label{keyprop}
    Let $P(T)=a_0T^n+a_1 T^{n-1}+\cdots+ a_n\in \bbL_p[T] $ be a polynomial with $a_n\neq 0$ and $\mathrm{Res}_P(T)\in\bar{\bbF}_p[T] $ its residue polynomial.
    Let $c\in \bar{\bbF}_p$ be a root of $\mathrm{Res}_P(T)$ with multiplicity $q$. We set
    $$P_{[c]}(T)=P(T+ [c]p^{s^P_{\max}})=\sum_{i=0}^{n}b_{n-i}T^i.$$Then we
    have: \begin{enumerate}
        \item $n-q$ is a breakpoint of $\Newt{P_{[c]}}$;
        \item in the range $\left[0,n-q\right]$,  $\Newt{P}$ is identical with $\Newt{P_{[c]}}$;
        \item the remaining slope(s) of $\Newt{P_{[c]}}$ are strictly greater than $s_{\max}^P$.
    \end{enumerate}
\end{proposition}
\begin{proof}Set $s=s_{\max}^P$ and $m=m_{\max}^{P}$.  Recall that, since $m$ is the maximal breakpoint of the Newton polygon $\Newt{P}$, for $m\leq n-k$, we have $$v_p(a_{n-k}p^{-v_p(a_m)-s(n-m-k)})=v_p(a_{n-k})-v_p(a_m)-s(n-m-k)\geq 0.$$
    Thus, $C_0\left(a_{n-k}p^{-v_t(a_m)-s(n-m-k)}\right)\in \bar{\bbF}_p$ is the image of $a_{n-k}p^{-v_p(a_m)-s(n-m-k)}$ under the canonical projection from $\bbL_p$ to its residue field $\bar{\bbF}_p$.

    Let $\mathrm{Res}_P(T+c)=\sum_{k=0}^{n-m}\kappa_{n-k}T^k$, then we have
    \begin{equation}\label{coeff_x_n_i}
        \begin{split}
            \kappa_{n-k}&= \sum_{i=n-k}^{n-m} C_0\left(a_{n-i}p^{-v_p(a_m)-s(n-m-i)}\right)\binom{i}{n-k}c^{i-(n-k)} \\ &= \sum_{j=0}^{k-m}C_0\left(p^{-v_p(a_m)-s(k-m)}a_{k-j}p^{sj}\right)\binom{n-k+j}{j}c^j    .
        \end{split}
    \end{equation}

    By the basic properties of the map $C_x(\alpha)$ (cf. \Cref{lemma_c_g}), we have
    \begin{align*}
        \text{\labelcref{coeff_x_n_i}}= & \sum_{j=0}^{k-m}\binom{n-k+j}{j}c^j  C_{v_p(a_m)+s(k-m)}\left( a_{k-j}p^{sj}\right)                                               \\
        =                               & \sum_{j=0}^{k-m}\binom{n-k+j}{j} C_{v_p(a_m)+s(k-m)}\left([c^j] a_{k-j}p^{sj}\right).\addtocounter{equation}{1}\tag{\theequation}
        \label{shit}
    \end{align*}
    A similar argument in the proof of \Cref{claim} in \Cref{lem:15792} shows that:
    $$v_p(a_{k-j})+sj\begin{cases}>v_p(a_m)+s(k-m),& \text{if } j>k-m \\ \geq v_p(a_m)+s(k-m),& \text{if } j\leq k-m\end{cases}.$$
    Again by \Cref{lemma_c_g},  for $j>k-m$, we have  $C_{v_p(a_m)+s(k-m)}\left([c^j]  a_{k-j}p^{sj}\right)=0$, and
    \begin{align*}
        \text{\Cref{shit}} & =\sum_{j=0}^{k}\binom{n-k+j}{j} C_{v_p(a_m)+s(k-m)}\left([c^j] a_{k-j}t^{sj}\right) \\
                           & =C_{v_p(a_m)+s(k-m)}\left(\sum_{j=0}^{k}\binom{n-k+j}{j} c^j a_{k-j}t^{sj}\right)   \\
                           & =C_{v_p(a_m)+s(k-m)}\left(b_k\right).
    \end{align*}

    Since $a_m=b_m$, one can conclude that,  for $0\leq k\leq n-m$, the coefficient $\kappa_{k}$ of $T^{n-k}$ in $\mathrm{Res}_P(T+c)$ equals to $C_{v_p(b_m)+s(k-m)}\left(b_k\right)$. Since $c$, as a root of $\mathrm{Res}_P(T)$, has multiplicity $q$, $T^q$ has non-zero coefficient in $\mathrm{Res}_P(T+c)$, i.e. we have  $$\kappa_{n-q}=C_{v_p(b_m)+s(n-q-m)}\left(b_k\right)(b_{n-q})\neq 0.$$ On the other hand, we have $v_p(b_{n-q})\geq v_p(b_m)+s(n-q-m)$. Thus, we have
    $$v_p(b_{n-q})= v_p(b_m)+s(n-q-m).$$
    If $k>n-q$, the coefficient $\kappa_{k}$ of $T^{n-k}$ in $\mathrm{Res}_P(T+c)$  is $0$, thus $v_p(b_{k})>v_t(b_m)+s(k-m)$.

\end{proof}
\begin{theorem}[Kedlaya]
    The field $\bbL_p$ is algebraic closed.
\end{theorem}
\begin{proof}
    Let $f(T)$ be a non-constant polynomial in $\bbL_p[T]$. To construct a root of $f(T)$ in $\bbL_p$ is equivalent to finding an element $r\in \bbL_p$ such that the maximal slope of the Newton polygon $\Newt{f(T+r)}$ is $\infty$.

    Let $\aleph_1$ be the minimal uncountable ordinal. We define a sequence of elements $(r_w)_{w<\aleph_1}$ in $\bbL_p$ by transfinite induction.
    \begin{itemize}
        \item For $w=0$, we set $r_0=0$ and let $\tau(0)$ be a root of $\mathrm{Res}_{f}(T)$ in $\bar{\mathbb{F}}_p$ and $d(0)$ be the maximal slope of the Newton polygon $\Newt{f(T+0)}$.
        \item Let $0<\omega<\aleph_1$ be a successor ordinal. If $f(r_{\omega-1})=0$, then we set $r_\omega=r_{\omega-1}=r_{\omega-1}+[0]\cdot p^\infty$, i.e. $\tau(\omega-1)=0$ and $d(\omega-1)=\infty$. Now suppose $f(r_{\omega-1})\neq 0$ and $r_\alpha$ has been constructed for any ordinal $\alpha<\omega$. Define $\tau(\omega-1)$ to be a root of $\mathrm{Res}_{f(T+r_{\omega-1})}(T)$ in $\bar{\bbF}_p$ and $d(\omega-1)$ to be the maximal slope of $\Newt{f(T+r_{\omega-1})}$. Then we set $r_{\omega}= r_{\omega-1}+  [\tau(\omega-1)]p^{d(\omega-1)}\in\bbL_p$.
        \item For limit ordinal $\omega_{\lim}$, we set $r_{\omega_{\lim}}=\sum_{\alpha<\omega_{\lim}}[\tau(\omega)]p^{d(\alpha)}\in \bbL_p$.
    \end{itemize}

    For all ordinals $\alpha<\beta$, by \Cref{keyprop}, we know that $d(\alpha)\leq d(\beta)$. Moreover, we have $d(\alpha)=d(\beta)$ if and only if $d(\omega)=\infty$ holds for all $\omega \geq \alpha$. Since there is no injection from $\aleph_1$ to $\bbQ$, there exist an ordinal $\phi<\aleph_1$ such that $d(\phi)=\infty$, i.e. $r_\phi$ is a root of $f(T)$ in $\bbL_p$.
\end{proof}
We summarize the above theorem into the following algorithm\footnote{This piece of pseudo-code dose not fit the definition of an algorithm in the sense of computer science, since it does not stop in finite steps and it does not treat the limit ordinal. We give this pseudo-code here to clarify the main ingredient of Kedlaya's proof.}:
\begin{algorithm}[H]
    \caption{transfinite Newton algorithm for $\bbL_p$}\label{pseudocode}
    \begin{algorithmic}
        \INPUT A non-constant polynomial $f(T)\in \bbL_p[T]$
        \OUTPUT A root of $f(T)$ in $\bbL_p$
        \Function{Newton}{$f$}
        \State $r\gets 0$
        \State $s_{\max}\gets 0,m_{\max}\gets 0,c\gets 0$
        \State $\mathrm{Res}_{\Phi}(T)\gets 0$
        \State $\Phi(T)\gets f(T)$ \Comment{We denote the coefficient of $T^i$ in $\Phi$ as $b_{n-i}$,where $n=\text{deg}(\Phi)$.}
        \While{$\Phi(0)\neq 0$}
        \State $m_{\max}\gets m_{\max}^{\Phi}$
        \State $s_{\max}\gets s_{\max}^{\Phi}$
        \State $\mathrm{Res}_{\Phi}(T)\gets\sum_{k=0}^{n-m_{\max}}C_{v_p(b_m)+s_{\max}(n-m_{\max}-k)}\left(b_{n-k}\right)T^k$
        \State $c\gets$ any root of $\mathrm{Res}_{\Phi}(T)$ in $\bar{\bbF}_p$
        \State $r\gets r+[c]\cdot p^{s_{\max}}$
        \State $\Phi(T)\gets \Phi(T+[c]\cdot p^{s_{\max}})$
        \EndWhile
        \State \Return $r$
        \EndFunction
    \end{algorithmic}
\end{algorithm}

\begin{definition}
    Let $l$ be a natural number. We call the value of $\Phi(T)$ (resp. $\mathrm{Res}_\Phi(T)$ and $r$) in the above pseudo-code after the loop iterates for $l$ times, the $l$-th approximation polynomial (resp. the $l$-th residue polynomial and the $l$-th approximation of a root of $\Phi(T)$).
\end{definition}

\section{Application of the transfinite Newton algorithm to the $p^n$-th cyclotomic polynomial}

\paragraph*{Convention} In this paragraph, we assume $n$ is a positive integer.

Let $\Phi_{p^n}(T)=\sum_{k=0}^{p-1}T^{p^{n-1}k}$ be the $p^n$-th cyclotomic polynomial, whose Newton polygon is a segment of slope $0$ with maximal breaking point $(0,0)$. In \Cref{statement}, we apply the transfinite Newton algorithm on $\Phi_{p^n}(T)$ to get the first $\aleph_0$ terms of a $p^n$-th primitive root in $\bbL_p$ and discuss the relations among all the posibilities of the first $\aleph_0$ terms of $p^n$-th primitive roots for $n\geq 2$. Besides that, we give an analogous expansion for $\zeta_p$. We establish our main combinatorial techniques in \Cref{Bellpoly} and complete the proof of \Cref{maintheorem1} in \Cref{Estimation}. Finally, using the expansion of $\zeta_{p^2}$, we give a uniformizer of $K_{2,m}$ for $m\geq 2$ in \Cref{subsec:42109}.

\subsection{First $\aleph_0$ terms of a root of $\Phi_{p^n}(T)$ }\label{statement}

The $0$-th residue polynomial of $\Phi_{p^n}(T)$ in the transfinite Newton algorithm is the polynomial $\frakA_{0,n}(T)= \sum_{k=0}^{p-1}T^{p^{n-1}k}$, and the canonical element $1\in\bar{\bbF}_p$ is a root of $\frakA_{0,n}(T)$. With these initial inputs, the first approximation polynomial is $$\Phi^{(1,n)}(T)=\sum_{k=0}^{p-1}(T+1)^{p^{n-1}k} ,$$which has $p^{n-1}(p-1)$ roots with the same valuation $v_p(\zeta_{p^n}-1)=\frac{1}{\varphi(p^n)}=\frac{1}{p^{n-1}(p-1)}>0$. As a consequence, the Newton polygon $\Newt{\Phi^{(1,n)}}$ is a segment of slope $\fraks_{1,n}=\frac{1}{p^{n-1}(p-1)}$ with maximal breakpoint $\frakm_{1,n}=(0,0)$. Thus, the first residue polynomial is
$$\frakA_{1,n}(T)=T^{p^{n-1}(p-1)}+1=(T^{p-1}+1)^{p^{n-1}}\in\overline{\bbF}_p[T],$$ which has $\frakz_{1,n}=(-1)^n\zetax\in \bar{\bbF}_p$ as a root\footnote{Adding the sign $(-1)^n$ to the root gives us a chance to get a compatible system $\left(\zeta_{p^n}\right)_{n\geq 1}$ of $p^n$-th primitive roots, i.e. $\zeta_{p^n}=\zeta_{p^{n+1}}^p$ (cf. \Cref{coro:54977}).} with multiplicity $\frakq_{1,n}=p^{n-1}$.

\begin{figure}[H]
    \centering
    \begin{tikzpicture}[every node/.style={scale=0.7},scale=1.5]
        \draw [<->] (0,2) -- (0,0) -- (5,0);
        \draw [thick] (0,0) -- (4.5,1.5);
        \node [below] at (0,0) {$0$};
        \node [below] at (4.5,0) {$p^{n-1}(p-1)$};
        \draw [dotted] (4.5,0) -- (4.5,1.5);
        \draw [dotted] (0,1.5) -- (4.5,1.5);
        \node [left] at (0,1.5) {$1$};
    \end{tikzpicture}
    \caption{$\Newt{\Phi^{(1,n)}}$}
\end{figure}

\begin{definition}\label{rmk:44996} The transfinite Newton algorithm can produce all the roots of a given polynomial $\Phi(T)\in \bbL_p[T]$. But note that we do not know whether the transfinite Newton algorithm can find all the roots of $\Phi_{p^n}(T)$,  for $n\geq 2$, by repeatedly choosing different roots of the residue polynomial in each step. We can get at least $p-1$ roots of $\Phi_{p^n}$ in this way since there are $p-1$ choices of $\frakz_{1,n}$ in $\bar{\bbF}_p$.

    For the roots of $\Phi_{p^n}(T)$ contructed by the transfinite Newton algorithm, the second assertion of \Cref{maintheorem1} below shows that their first $\aleph_0$ terms have exactly $p-1$ possibilities which are completely determined by the choice of $\frakz_{1,n}$ in $\bar{\bbF}_p$. Moreover, for any root of $\Phi_{p^n}(T)$ , there exists a root constructed by the algorithm such that they have the same first $\aleph_0$ terms.
\end{definition}

Let $\zeta_{p^n}$ be a root of $\Phi_{p^n}(T)$ constructed by the transfinite Newton algorithm with input $\frakz_{1,n}=(-1)^n\zetax$. The following proposition summarizes the above discussion for the initial terms.
\begin{proposition}\label{exp_first_prop}\label{exp_first_rmk}
    One has:
    \begin{enumerate}
        \item $\fraks_{0,n}=0\in \bbQ$, $\frakz_{0,n}=1\in \bar{\bbF}_p$;
        \item $\fraks_{1,n}=\frac{1}{p^{n-1}(p-1)}$, $\frakz_{1,n}=(-1)^n\zetax\in \bar{\bbF}_p$ and the multiplicity of $\frakz_{1,n}=(-1)^n\zetax$ in $\frakA_1(T)$ is $\frakq_{1,n}=p^{n-1}$.
    \end{enumerate}
    In conclusion, the first approximation of $\zeta_{p^n}$ is $\zeta_{p^n}^{(1)}=\Lambda_{1,n}$, where $\Lambda_{1,n}$ is defined to be $1+(-1)^n\zetax p^{\frac{1}{p^{n-1}(p-1)}}$. In other words, we have
    $$\zeta_{p^n}=1+(-1)^n\zetax p^{\frac{1}{p^{n-1}(p-1)}}+o\left(p^{\frac{1}{p^{n-1}(p-1)}}\right).$$
\end{proposition}

For every $\alpha\in\bbL_p$, denote by $\aleph_0(\alpha)$ the first $\aleph_0$ terms of the expansion of $\alpha$ in $\bbL_p$.
The following theorem gives the explicit formula for $\aleph_0\left(\zeta_{p^n}\right)$ with $n\geq 2$:
\begin{theorem}\label{maintheorem1}\label{thm:galcong}
    Let $n\geq 2$ be an integer.
    \begin{enumerate}
        \item Let $\zeta_{p^n}^{(i)}$ be the $i$-th approximation of $\zeta_{p^n}$ in the transfinite Newton algorithm. Then we have
              $$\zeta_{p^n}^{(i)}=
                  \begin{cases}
                      \sum_{k=0}^i \frac{(-1)^{kn}}{[k!]}\zetax^kp^{\frac{k}{p^{n-1}(p-1)}},                        & \text{ for } 0\leq i\leq p-1, \\
                      \zeta_{p^n}^{(p-1)}+ \sum_{l=n}^{i-p+n}(-1)^n\zetax p^{\frac{1}{p^{n-2}(p-1)}-\frac{1}{p^l}}, & \text{ for }  i\geq p,
                  \end{cases}
              $$
              i.e.
              $$\aleph_0\left(\zeta_{p^n}\right)=\sum_{i=0}^{p-1}\frac{\left((-1)^n\zetax\right)^i}{[i!]} p^{\frac{i}{p^{n-1}(p-1)}}+\sum_{j=n}^\infty (-1)^n\zetax p^{\frac{1}{p^{n-2}(p-1)}-\frac{1}{p^j}}.$$
        \item There exists $p-1$ distinct elements $m_0,\cdots,m_{p-2}$ in $\left(\bbZ/p^n\bbZ\right)^\times$ such that
              \begin{enumerate}
                  \item $$\aleph_0\left(\zeta_{p^n}^{m_k}\right)=\sum_{i=0}^{p-1}\frac{\left((-1)^n\zetax^{2k+1}\right)^i}{[i!]} p^{\frac{i}{p^{n-1}(p-1)}}+\sum_{j=n}^\infty (-1)^n\zetax^{2k+1} p^{\frac{1}{p^{n-2}(p-1)}-\frac{1}{p^j}};$$
                  \item $\calR_n=\{m_0,\cdots,m_{p-2}\}$ forms a $\mathrm{mod}\ p$ residue system of $\left(\bbZ/p^n\bbZ\right)^\times$.
              \end{enumerate}
        \item For every $m\in\left(\bbZ/p^n\bbZ\right)^\times$, there exists a unique $m_t\in\calR_n$ such that $\aleph_0\left(\zeta_{p^n}^m\right)=\aleph_0\left(\zeta_{p^n}^{m_t}\right)$.
    \end{enumerate}

\end{theorem}
\begin{proof}
    \begin{enumerate}
        \item

              We sketch the proof of the first assertion and leave the technical  details of each step in next sections.
              We denote the formula on the right-hand side of the theorem by $\Lambda_{i,n}$ and the $i$-th approximation polynomial of $\zeta_{p^n}$ by
              \begin{equation}\Phi^{(i,n)}(T)=\Phi_{p^n}\left(T+\zeta_{p^n}^{(i-1)}\right)=\sum_{k=0}^{p^{n-1}(p-1)}b_{p^{n-1}(p-1)-k}^{(i,n)}T^k\in \bbL_p[T].\end{equation}
              Moreover, we denote by $\frakA_{i,n}(T)\in \bar{\bbF}_p[T]$ the residue polynomial of $\Phi^{(i,n)}(T)$.

              By the transfinite Newton algorithm, it is crucial to determine the following data:
              \begin{enumerate}
                  \item The maximal slope $\fraks_{i,n}$ of the Newton polygon of $\Phi^{(i,n)}(T)$, which gives the support of the desired expansion,
                  \item The residue polynomial $\frakA_{i,n}(T)$, whose root $\frakz_{i,n}\in\bar{\bbF}_p$ with multiplicity $\frakq_{i,n}$ gives the coefficient $\alpha_{\fraks_{i,n}}$ of the desired expansion.
              \end{enumerate}

              To prove the theorem, one only needs to check that the supports and the coefficients in the $i$-th step do coincide with those of $\Lambda_{i,n}$. The strategy of the proof is the following:
              \begin{enumerate}[listparindent=\parindent]
                  \item \textbf{Describe the initial terms} (cf. \Cref{exp_first_prop}): in fact, we have $\fraks_{0,n}=0, \fraks_{1,n}=\frac{1}{p^{n-1}(p-1)}$, $\frakz_{0,n}=1$ and $\frakz_{1,n}=(-1)^n\zetax$ with multiplicity $\frakq_{1,n}=p^{n-1}$.
                  \item \textbf{Induction on $i$ for  $2\leq i\leq p-1$}. Assume that, for $1\leq j\leq i-1$, we have $\zeta_{p^n}^{(j)}=\Lambda_{j,n}$. In other words, the maximal slope $\fraks_{j,n}$ of the Newton polygon $\Newt{\Phi^{(j,n)}}$ of the $j$-th approximation polynomial  is $\frac{j}{p^{n-1}(p-1)}$ and the $j$-th residue polynomial $\frakA_{j,n}(T)$ has a root $\frakz_{j,n}=\frac{(-1)^{jn}}{j!}\zetax^j\in \bar{\bbF}_p$ with multiplicity $\frakq_{j,n}=p^{n-1}$. We describe the Newton polygon of the $i$-th approximation polynomial $\Phi^{(i,n)}(T)$ as follows.
                        \par By the induction hypothesis and \Cref{keyprop}, for  $1\leq j\leq i-1$, the Newton polygon of $\Phi^{(j,n)}(T)$ and $\Phi^{(j+1,n)}(T)$ are identical in the range $[0,p^{n-1}(p-1)-\frakq_{j}]$, i.e. $[0,p^{n-1}(p-2)]$. Therefore, the Newton polygons $\Newt{\Phi^{(i,n)}}$ and $\Newt{\Phi^{(1,n)}}$ are identical in the range $[0,p^{n-1}(p-2)]$, and $p^{n-1}(p-2)$ is a breakpoint of the Newton polygon $\Newt{\Phi^{(i,n)}}$. As a result, we only need to consider $\Newt{\Phi^{(i,n)}}$ in the range $[p^{n-1}(p-2), p^{n-1}(p-1)]$. In other words, we need to estimate the $p$-adic valuation of $b^{(i,n)}_{p^{n-1}(p-1)-k}$ for $0\leq k\leq p^{n-1}$.
                        \par By the transfinite Newton algorithm and the assumption $\zeta_{p^n}^{(i-1)}=\Lambda_{i-1,n}$, we can obtain the formula for the coefficients $b_{p^{n-1}(p-1)-k}^{(i,n)}$ of the $i$-th approximation polynomials $\Phi^{(i,n)}$ with $0\leq k\leq p^{n-1}$ and their $p$-adic valuation can be calculated by the estimation of the $p$-adic valuation of $\Lambda_{i-1,n}^{p^{n-1}}-1$ and $\Lambda_{i-1,n}^{p^{n}}-1$ established in \Cref{Estimation} (cf. \Cref{premier} and \Cref{seconde} respectively), which relies on the arithmetic properties of incomplete exponential Bell polynomial studied in \Cref{Bellpoly}.

                        \begin{enumerate}
                            \item If $k=0$, we have
                                  $$b_{p^{n-1}(p-1)}^{(i,n)}=\sum_{l=0}^{p-1}\Lambda_{i-1,n}^{lp^{n-1}}=\frac{\Lambda_{i-1,n}^{p^n}-1}{\Lambda_{i-1,n}^{p^{n-1}}-1}. $$
                                  By \Cref{premier} and \Cref{seconde},  we have
                                  \begin{align*}
                                      b_{p^{n-1}(p-1)}^{(i,n)} & = \frac{\frac{(-1)^{i-1}}{i!}\zetax^{i}p^{1+\frac{i}{p-1}}+o\left(p^{1+\frac{i}{p-1}}\right)}{\sum_{l=1}^{i-1}\frac{(-1)^l}{[l!]}\zetax^l p^{\frac{l}{p-1}}+O\left(p^{1+\frac{1}{p(p-1)}}\right)} \\
                                                               & = \frac{\frac{(-1)^{i-1}}{i!}\zetax^{i}p^{1+\frac{i}{p-1}}+o\left(p^{1+\frac{i}{p-1}}\right)}{-\zetax p^{\frac{1}{p-1}}+O\left(p^{\frac{2}{p-1}}\right)}                                          \\
                                                               & =  \frac{(-1)^i}{i!}\zetax^{i-1} p^{1+\frac{i-1}{p-1}}+o\left(p^{1+\frac{i-1}{p-1}}\right) .
                                  \end{align*}

                            \item If $1\leq k\leq p^{n-1}$, we have
                                  \begin{align*}
                                      b_{p^{n-1}(p-1)-k}^{(i,n)} & =\sum_{l=1}^{p-1}\binom{p^{n-1}l}{k}\Lambda_{i-1,n}^{p^{n-1}l-k}                                                                                                                                                                                                                    \\                                                                                                                 &=\sum_{l=1}^{p-1}\left(p^{n-1}l\frac{(-1)^{k-1}}{k}+O(p^n)\right)\Lambda_{i-1,n}^{p^{n-1}l-k} \\
                                                                 & =\frac{(-1)^{k-1}p^{n-1}}{k\Lambda_{i-1,n}^k}\sum_{l=1}^{p-1}l\Lambda_{i-1,n}^{p^{n-1}l}+O(p^n)                                                                                                                                                                                   .
                                  \end{align*}
                                  Together with the elementary identity $\sum_{l=1}^{p-1}l\Lambda_{i-1,n}^{p^{n-1}l}=\frac{p\Lambda_{i-1,n}^{p^n}}{\Lambda_{i-1,n}^{p^{n-1}}-1}-\Lambda_{i-1,n}^{p^{n-1}}\frac{\Lambda_{i-1,n}^{p^{n}}-1}{(\Lambda_{i-1,n}^{p^{n-1}}-1)^2}$, we obtain
                                  $$b_{p^{n-1}(p-1)-k}^{(i,n)}=\frac{(-1)^{k-1}p^{n-1}}{k\Lambda_{i-1,n}^k}\left(\frac{p\Lambda_{i-1,n}^{p^n}}{\Lambda_{i-1,n}^{p^{n-1}}-1}-\Lambda_{i-1,n}^{p^{n-1}}\frac{\Lambda_{i-1,n}^{p^{n}}-1}{(\Lambda_{i-1,n}^{p^{n-1}}-1)^2}\right)+O(p^n) .$$
                                  One can use again the estimation of the $p$-adic valuation of $\Lambda_{i-1,n}^{p^n}-1$ and $\Lambda_{i-1,n}^{p^{n-1}}-1$ in \Cref{Estimation} to deduce that
                                  \begin{equation}
                                      \begin{split}
                                          v_p\left(b_{p^{n-1}(p-1)-k}^{(i,n)}\right)&=v_p\left(\frac{(-1)^{k-1}p^{n-1}}{k\Lambda_{i-1,n}^k}\frac{p\Lambda_{i-1,n}^{p^n}}{\Lambda_{i-1,n}^{p^{n-1}}-1}\right)\\
                                          &=
                                          \begin{cases}
                                              n-v_p(k)-\frac{1}{p-1}\geq 1+\frac{i-1}{p-1}, & 1\leq k<p^{n-1}; \\
                                              1- \frac{1}{p-1},                             & k=p^{n-1},
                                          \end{cases}\label{val_small}
                                      \end{split}
                                  \end{equation}
                                  and
                                  \begin{align*}
                                      C_{\frac{p-2}{p-1}}\left(b_{p^{n-1}(p-1)-p^{n-1}}^{(i,n)}\right) & =C_{\frac{p-2}{p-1}}\left(\frac{(-1)^{p-1}p^{n-1}}{p^{n-1}\Lambda_{i-1,n}^{p^{n-1}}}\frac{p\Lambda_{i-1,n}^{p^n}}{\Lambda_{i-1,n}^{p^{n-1}}-1}\right)
                                      =-\zetax^{-1} .\addtocounter{equation}{1}\tag{\theequation}\label{coeff_p_p_2}
                                  \end{align*}

                        \end{enumerate}

                        In conclusion, the Newton polygon of the $i$-th approximation polynomial $\Phi^{(i,n)}$ has three breakpoints: $0, p^{n-1}(p-2)$ and $p^{n-1}(p-1)$, with maximal breakpoint $\frakm_{i,n}=p^{n-1}(p-2)$ and maximal slope
                        $$\fraks_{i,n}=\frac{v_p\left(b_{p^{n-1}(p-1)}^{(i,n)}\right)-v_p\left(b_{p^{n-1}(p-2)}^{(i,n)}\right)}{p^{n-1}(p-1)-p^{n-1}(p-2)}=\frac{i}{p^{n-1}(p-1)}.$$
                        The $i$-th residue polynomial
                        $$\frakA_{i,n}(T)=-\zetax^{-1}T^{p^{n-1}}+\frac{(-1)^i}{i!}\zetax^{i-1}$$
                        has $\frakz_{i,n}=\frac{(-1)^{in}}{i!}\zetax^i$ as a root with multiplicity $\frakq_r=p^{n-1}$.
                        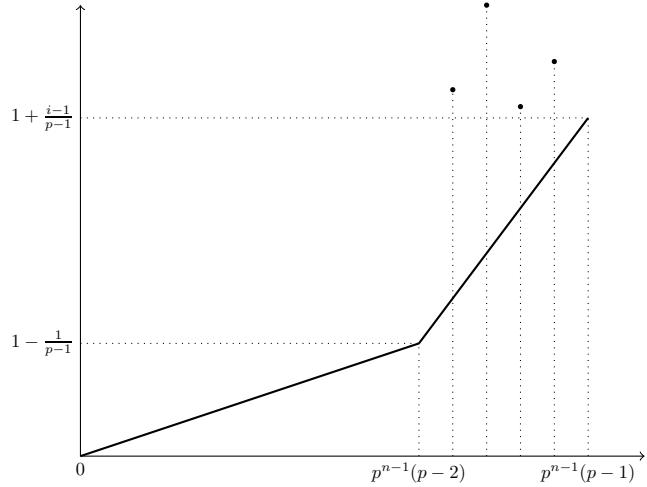
\begin{figure}[H]
                            \centering
                            \begin{tikzpicture}[every node/.style={scale=0.7},scale=1.5]
                                \draw [<->] (0,4) -- (0,0) -- (5,0);
                                \draw [thick] (0,0) -- (3,1) -- (4.5,3);
                                \draw [dotted] (3,0) -- (3,1);
                                \draw [dotted] (4.5,0) -- (4.5,3);
                                \draw [dotted] (0,1) -- (3,1);
                                \draw [dotted] (0,3) -- (4.5,3);
                                \draw [dotted] (3.3,0) -- (3.3,3.25);
                                \draw [dotted] (3.6,0) -- (3.6,4);
                                \draw [dotted] (3.9,0) -- (3.9,3.1);
                                \draw [dotted] (4.2,0) -- (4.2,3.5);
                                \node [below] at (0,0) {$0$};
                                \node [below] at (3,0) {$p^{n-1}(p-2)$};
                                \node [below] at (4.5,0) {$p^{n-1}(p-1)$};
                                \node [left] at (0,1) {$1-\frac{1}{p-1}$};
                                \node [left] at (0,3) {$1+\frac{i-1}{p-1}$};
                                \node [circle,fill,inner sep=1pt] at (3.3,3.25) {};
                                \node [circle,fill,inner sep=1pt] at (3.6,4) {};
                                \node [circle,fill,inner sep=1pt] at (3.9,3.1) {};
                                \node [circle,fill,inner sep=1pt] at (4.2,3.5) {};
                            \end{tikzpicture}
                            \caption{$\Newt{\Phi^{(i,n)}}$, $2\leq i\leq p-1$}
                        \end{figure}

                  \item \textbf{Induction on $i\geq p$ for $\zeta_{p^n}^{(i)}$}.
                        For the initial term $i=p$,
                        the transfinite Newton algorithm and the results proved in the previous steps imply that:
                        \begin{enumerate}
                            \item The Newton polygons $\Newt{\Phi^{(p,n)}}$ and $\Newt{\Phi^{(p-1,n)}}$ are identical in the range $[0, p^{n-1}(p-2)]$;
                            \item $p^{n-1}(p-2)$ is a breakpoint of $\Newt{\Phi^{(p,n)}}$.
                        \end{enumerate}
                        Therefore, we only need to consider $\Newt{\Phi^{(p,n)}}$ in the range $[p^{n-1}(p-2), p^{n-1}(p-1)]$, i.e. to estimate the $p$-adic valuation of $b_{p^{n-1}(p-1)-k}^{(p,n)}$ for $0\leq k\leq p^{n-1}$. We express $b_{p^{n-1}(p-1)-k}^{(p,n)}$ in terms of $\Lambda_{p-1,n}$ as following:
                        \begin{gather*}
                            b_{p^{n-1}(p-1)-k}^{(p,n)}=
                            \begin{cases}
                                \sum_{l=0}^{p-1}\Lambda_{p-1,n}^{lp^{n-1}}=\frac{\Lambda_{p-1,n}^{p^n}-1}{\Lambda_{p-1,n}^{p^{n-1}}-1},                                                                                                                & \text{ if } k=0;                 \\
                                \frac{(-1)^{k-1}p^{n-1}}{k\Lambda_{p-1,n}^k}\left(\frac{p\Lambda_{p-1,n}^{p^n}}{\Lambda_{p-1,n}^{p^{n-1}}-1}-\Lambda_{p-1,n}^{p^{n-1}}\frac{\Lambda_{p-1,n}^{p^{n}}-1}{(\Lambda_{p-1,n}^{p^{n-1}}-1)^2}\right)+O(p^n), & \text{ if } 1\leq k\leq p^{n-1}.
                            \end{cases}
                        \end{gather*}
                        Again using the estimation of the $p$-adic valuation of $\Lambda_{p-1,n}^{p^{n-1}}-1$ and $\Lambda_{p-1,n}^{p^{n}}-1$ in \Cref{Estimation}, we have
                        \begin{enumerate}
                            \item $$b_{p^{n-1}(p-1)}^{(p,n)}=\frac{\zetax p^{2+\frac{1}{p-1}-\frac{1}{p}}+O\left(p^{2+\frac{1}{p-1}}\right)}{-\zetax p^{\frac{1}{p-1}}+O\left(p^{\frac{2}{p-1}}\right)}=-p^{2-\frac{1}{p}}+o\left(p^{2-\frac{1}{p}}\right) ,$$
                            \item
                                  \begin{gather*}
                                      v_p\left(b_{p^{n-1}(p-1)-k}^{(p,n)}\right)=n-v_p(k)-v_p\left(\Lambda_{p-1,n}^{p^{n-1}}-1\right)=n-v_p(k)-\frac{1}{p-1} \text{ for } k=1,\cdots,p^{n-1}-1 ,
                                  \end{gather*}
                            \item $$b_{p^{n-1}(p-2)}^{(p,n)}=(1+o(1))\frac{(-1)^{p^{n-1}-1}p}{-\zetax p^{\frac{1}{p-1}}+O\left(p^{\frac{2}{p-1}}\right)}=-\zetax^{-1}p^{\frac{p-2}{p-1}}+o\left(p^{\frac{p-2}{p-1}}\right) .$$
                        \end{enumerate}

                        In other words, the valuation of $b_{p^{n-1}(p-1)-k}^{(p,n)}$ is given by
                        $$v_p\left(b_{p^{n-1}(p-1)-k}^{(p,n)}\right)=
                            \begin{cases}
                                2-\frac{1}{p},          & \text{ if } k=0;             \\
                                n-v_p(k)-\frac{1}{p-1}, & \text{ if } 1\leq k<p^{n-1}; \\
                                \frac{p-2}{p-1},        & \text{ if } k=p^{n-1},
                            \end{cases}
                        $$
                        and the coefficient of $b_{p^{n-1}(p-1)}^{(p,n)}$ at $2-\frac{1}{p}$ equals $-1$, the coefficient of $b_{p^{n-1}(p-2)}^{(p,n)}$ at $\frac{p-2}{p-1}$ equals $-\zetax^{-1}$.
                        \par Notice that the segment $L_{p,n}$ with endpoints
                        $$\left(p^{n-1}(p-2),v_p\left(b_{p^{n-1}(p-2)}^{(p,n)}\right)\right)=\left(p^{n-1}(p-2),\frac{p-2}{p-1}\right)$$
                        and
                        $$\left(p^{n-1}(p-1),v_p\left(b_{p^{n-1}(p-1)}^{(p,n)}\right)\right)=\left(p^{n-1}(p-1),\frac{2p-1}{p}\right)$$
                        has slope
                        $$\frac{1}{p^{n-1}}\left(2-\frac{1}{p}-\frac{p-2}{p-1}\right)=\frac{1}{p^{n-2}(p-1)}-\frac{1}{p^n} $$
                        and,  for all $k\in\{1,2,\cdots,p^{n-1}-1\}$,
                        $$n-v_p(k)-\frac{1}{p-1}\geq \frac{p-2}{p-1}+\left(\frac{1}{p^{n-2}(p-1)}-\frac{1}{p^n}\right)\left(\left(p^{n-1}(p-1)-k\right)-p^{n-1}(p-2)\right).$$
                        In conclusion, $L_{p,n}$ is the segment of the Newton polygon $\Newt{\Phi^{(p,n)}}$ with maximal slope $\fraks_{p,n}=\frac{1}{p^{n-2}(p-1)}-\frac{1}{p^n}$. Therefore, we have
                        $$\frakA_{p,n}(T)=-\zetax^{-1}T^{p^{n-1}}-1,$$
                        which has $\frakz_{p,n}=(-1)^n\zetax$ as a root with multiplicity $\frakq_{p,n}=p^{n-1}$.
                        \par Now let $i\geq p+1$. Suppose, for all $2\leq l\leq i-1$, the theorem holds. i.e. we have $\zeta_{p^n}^{(i-1)}=\Lambda_{i-1,n}$ and $\frakq_{i-1,n}=p$. Similar to the previous case, by induction we may assume $\Newt{\Phi^{(i,n)}}$ and $\Newt{\Phi^{(1,n)}}$ are identical in the range $[0, p^{n-1}(p-2)]$. Therefore, we are reduced to consider $\Newt{\Phi^{(i,n)}}$ in the range $[p^{n-1}(p-2), p^{n-2}(p-1)]$.

                        \par By the induction hypothesis $\zeta_{p^n}^{(i-1)}=\Lambda_{i-1,n}$, we get
                        \begin{gather*}
                            b_{p^{n-1}(p-1)-k}^{(i,n)}=
                            \begin{cases}
                                \sum_{l=0}^{p-1}\Lambda_{i-1,n}^{lp^{n-1}}=\frac{\Lambda_{i-1,n}^{p^n}-1}{\Lambda_{i-1,n}^{p^{n-1}}-1},                                                                                                                & \text{ if } k=0;                 \\
                                \frac{(-1)^{k-1}p^{n-1}}{k\Lambda_{i-1,n}^k}\left(\frac{p\Lambda_{i-1,n}^{p^n}}{\Lambda_{i-1,n}^{p^{n-1}}-1}-\Lambda_{i-1,n}^{p^{n-1}}\frac{\Lambda_{i-1,n}^{p^{n}}-1}{(\Lambda_{i-1,n}^{p^{n-1}}-1)^2}\right)+O(p^n), & \text{ if } 1\leq k\leq p^{n-1}.
                            \end{cases}
                        \end{gather*}
                        Again using the estimation of the $p$-adic valuation of $\Lambda_{i-1,n}^{p^n}-1$ and $\Lambda_{i-1,n}^{p^{n-1}}-1$ in \Cref{Estimation}, we have
                        \begin{enumerate}
                            \item $$b_{p^{n-1}(p-1)}^{(i,n)}=\frac{\zetax p^{2+\frac{1}{p-1}-\frac{1}{p^{i-p+1}}}+O\left(p^{2+\frac{1}{p-1}}\right)}{-\zetax p^{\frac{1}{p-1}}+O\left(p^{\frac{2}{p-1}}\right)}=-p^{2-\frac{1}{p^{i-p+1}}}+o\left(p^{2-\frac{1}{p^{i-p+1}}}\right) ,$$
                            \item
                                  \begin{gather*}
                                      v_p\left(b_{p^{n-1}(p-1)-k}^{(i,n)}\right)=n-v_p(k)-v_p\left(\Lambda_{p-1,n}^{p^{n-1}}-1\right)=n-v_p(k)-\frac{1}{p-1} \text{, for } k=1,\cdots,p^{n-1}-1 ,
                                  \end{gather*}
                            \item $$b_{p^{n-1}(p-2)}^{(i,n)}=(1+o(1))\frac{(-1)^{p^{n-1}-1}p}{-\zetax p^{\frac{1}{p-1}}+O\left(p^{\frac{2}{p-1}}\right)}=-\zetax^{-1}p^{\frac{p-2}{p-1}}+o\left(p^{\frac{p-2}{p-1}}\right) .$$
                        \end{enumerate}
                        In other words, the valuation of $b_{p^{n-1}(p-1)-k}^{(i,n)}$ is given by
                        $$v_p\left(b_{p^{n-1}(p-1)-k}^{(i,n)}\right)=
                            \begin{cases}
                                2-\frac{1}{p^{i-p+1}}   & \text{ if } k=0;             \\
                                n-v_p(k)-\frac{1}{p-1}, & \text{ if } 1\leq k<p^{n-1}; \\
                                \frac{p-2}{p-1},        & \text{ if }k=p^{n-1},
                            \end{cases}
                        $$
                        and the coefficient of $b_{p^{n-1}(p-1)}^{(i,n)}$ at $2-\frac{1}{p^{i-p+1}}$ equals to $-1$, the coefficient of $b_{p^{n-1}(p-1)}^{(i,n)}$ at $\frac{p-2}{p-1}$ equals to $-\zetax^{-1}$. Notice that the segment $L_{i,n}$ with endpoints
                        $$\left(p^{n-1}(p-2),v_p\left(b_{p^{n-1}(p-2)}^{(i,n)}\right)\right)=\left(p^{n-1}(p-2),\frac{p-2}{p-1}\right)$$
                        and
                        $$\left(p^{n-1}(p-1),v_p\left(b_{p^{n-1}(p-1)}^{(i,n)}\right)\right)=\left(p^{n-1}(p-1),2-\frac{1}{p^{i-p+1}}\right)$$
                        has slope
                        $$\frac{1}{p^{n-1}}\left(2-\frac{1}{p^{i-p+1}}-\frac{p-2}{p-1}\right)=\frac{1}{p^{n-2}(p-1)}-\frac{1}{p^{i-p+n}} $$
                        and,  for all $k\in\{1,2,\cdots,p^{n-1}-1\}$,
                        $$n-v_p(k)-\frac{1}{p-1}\geq \frac{p-2}{p-1}+\left(\frac{1}{p^{n-2}(p-1)}-\frac{1}{p^{i-p+n}}\right)\left(\left(p^{n-1}(p-1)-k\right)-p^{n-1}(p-2)\right).$$
                        We conclude that $L_{i,n}$ is the segment of $\Newt{\Phi^{(i,n)}}$ with maximal slope $$\fraks_{i,n}=\frac{1}{p^{n-2}(p-1)}-\frac{1}{p^{i-p+n}} .$$ Therefore, we have
                        $$\frakA_{i,n}(T)=-\zetax^{-1}T^{p^{n-1}}-1,$$
                        which has $\frakz_{i,n}=(-1)^n\zetax$ as a root with multiplicity $\frakq_{i,n}=p^{n-1}$.
                        \begin{figure}[H]
                            \centering
                            \begin{tikzpicture}[every node/.style={scale=0.7},scale=1.5]
                                \draw [<->] (0,4) -- (0,0) -- (5,0);
                                \draw [thick] (0,0) -- (3,1) -- (4.5,3.2);
                                \draw [dotted] (3,0) -- (3,1);
                                \draw [dotted] (4.5,0) -- (4.5,3.2);
                                \draw [dotted] (0,1) -- (3,1);
                                \draw [dotted] (0,3.2) -- (4.5,3.2);
                                \draw [dotted] (3.3,0) -- (3.3,3.25);
                                \draw [dotted] (3.6,0) -- (3.6,4);
                                \draw [dotted] (3.9,0) -- (3.9,3);
                                \draw [dotted] (4.2,0) -- (4.2,3.5);
                                \node [below] at (0,0) {$0$};
                                \node [below] at (3,0) {$p^{n-1}(p-2)$};
                                \node [below] at (4.5,0) {$p^{n-1}(p-1)$};
                                \node [left] at (0,1) {$\frac{p-2}{p-1}$};
                                \node [left] at (0,3.2) {$2-\frac{1}{p^{i-p+1}}$};
                                \node [circle,fill,inner sep=1pt] at (3.3,3.25) {};
                                \node [circle,fill,inner sep=1pt] at (3.6,4) {};
                                \node [circle,fill,inner sep=1pt] at (3.9,3) {};
                                \node [circle,fill,inner sep=1pt] at (4.2,3.5) {};
                            \end{tikzpicture}
                            \caption{$\Newt{\Phi^{(i,n)}}$, $i\geq p$}
                        \end{figure}

              \end{enumerate}
        \item The proof of the first assertion shows that $\aleph_0\left(\zeta_{p^n}\right)$ is completely determined once we have fixed our choice of $\frakz_{1,n}$. As we mentioned in \Cref{rmk:44996}, there are $p-1$ different candidates for $\frakz_{1,n}$: $(-1)^n\zetax^{2k+1}, k=0,1,\cdots,p-2$. Therefore there exist $p-1$ different elements $m_0,\cdots,m_{p-2}\in\left(\bbZ/p^n\bbZ\right)^\times$ such that $$\aleph_0\left(\zeta_{p^n}^{m_k}\right)=\sum_{i=0}^{p-1}\frac{\left((-1)^n\zetax^{2k+1}\right)^i}{[i!]} p^{\frac{i}{p^{n-1}(p-1)}}+\sum_{j=n}^\infty (-1)^n\zetax^{2k+1} p^{\frac{1}{p^{n-2}(p-1)}-\frac{1}{p^j}}.$$
              If there exist two different elements $m_{t_1},m_{t_2}$ in $\calR_n$ such that $m_{t_1}\equiv m_{t_2} \pmod{p}$, then by abuse of notations, we assume $m_{t_1},m_{t_2}\in\bbZ$ and $m_{t_1}=ph+m_{t_2}$, where $h$ is a positive integer. Then we have $\zeta_{p^n}^{m_{t_1}}=\zeta_{p^n}^{m_{t_2}}\cdot \left(\zeta_{p^n}^p\right)^h$, where $\left(\zeta_{p^n}^p\right)^h$ is a $p^{n-1}$-th root of unity. Set $\left(\zeta_{p^n}^p\right)^h=\zeta_{p^{n-1}}^r$ with $r$ a positive integer. By \Cref{exp_first_prop} and \Cref{exp_first_rmk}, we have $\zeta_{p^{n-1}}=1+O\left(p^{\frac{1}{p^{n-2}(p-1)}}\right)$ for all $n\geq 2$. Therefore
              $$\zeta_{p^{n-1}}^r=\left(1+O\left(p^{\frac{1}{p^{n-2}(p-1)}}\right)\right)^r=1+O\left(p^{\frac{1}{p^{n-2}(p-1)}}\right),$$
              and consequently
              $$\zeta_{p^n}^{m_{t_1}}=\zeta_{p^n}^{m_{t_2}}\cdot \left(1+O\left(p^{\frac{1}{p^{n-2}(p-1)}}\right)\right)=\zeta_{p^n}^{m_{t_2}}+O\left(p^{\frac{1}{p^{n-2}(p-1)}}\right),$$
              i.e. $\aleph_0\left(\zeta_{p^n}^{m_{t_1}}\right)=\aleph_0\left(\zeta_{p^n}^{m_{t_2}}\right)$, which contradicts our assumption. Now the result follows from the fact that there are exactly $p-1$ $\mathrm{mod}\ p$ residue classes in $\left(\bbZ/p^n\bbZ\right)^\times$.
        \item Similar to the proof of the second assertion, we can prove that for $\tilde{m}\in\left(\bbZ/p^n\bbZ\right)^\times$, $\aleph_0\left(\zeta_{p^n}^m\right)=\aleph_0\left(\zeta_{p^n}^{\tilde{m}}\right)$ is equivalent to $m\equiv \tilde{m}\pmod{p}$. By the second assertion, we can find such $\tilde{m}$ in $\calR_n$.
    \end{enumerate}
\end{proof}

Instead of using the Newton algorithm directly, we explore the canonical expansion of $\zeta_p$ in $\bbL_p$ by using the expansion of $\zeta_{p^2}$:
\begin{proposition}\label{prop:1238}
    The canonical expansion of $\zeta_p$ is given as following: $\zeta_{p}=\sum_{k=0}^\infty [c_k]p^{\frac{k}{p-1}}$, with $c_k\in \bbF_{p^2}$ for all $k\in\bbZ_{\geq 0}$. In particular, for $0\leq k\leq p-1$, we have $c_k=(-1)^k\frac{\zetax^k}{k!}$.
\end{proposition}
\begin{proof}
    The first assertion, as a direct consequence of \Cref{lem:tame_iso}, is proved by Lampert (cf. \cite{Lampert1986}).
    Since $\zeta_{p^2}^p$ is a primitive $p$-th root, we may assume $\zeta_{p^2}^p=\zeta_p^r$ for some $r\in\{1,2,\cdots,p-1\}$. On the one hand, we calculate
    \begin{equation}\label{eq:50544}
        \begin{aligned}
            \left(\zeta_{p^2}\right)^p= & \left(\sum_{k=0}^{p-1}\frac{\zetax^k}{[k!]}p^{\frac{1}{p(p-1)}}+O\left(p^{\frac{1}{p-1}-\frac{1}{p^2}}\right)\right)^p \\
            =                           & \sum_{k=0}^{p-1}\left(\frac{\zetax^k}{[k!]}p^{\frac{1}{p(p-1)}}\right)^p+o(p)                                          \\
            =                           & \sum_{k=0}^{p-1}(-1)^k\frac{\zetax^k}{[k!]}p^{\frac{k}{p-1}}+o(p) .
        \end{aligned}
    \end{equation}
    On the other hand, since $\zeta_p=1-\zetax p^{\frac{1}{p-1}}+o\left(p^{\frac{1}{p-1}}\right)$ (cf. \Cref{exp_first_rmk} of local cite), we have
    \begin{equation}\label{eq:56576}
        \zeta_p^r=\left(1-\zetax p^{\frac{1}{p-1}}+o\left(p^{\frac{1}{p-1}}\right)\right)^r=1-[r]\zetax p^{\frac{1}{p-1}}+o\left(p^{\frac{1}{p-1}}\right) .
    \end{equation}
    By comparing \Cref{eq:50544} and \Cref{eq:56576}, we know that $r=1$ and consequently
    $$\zeta_p=\zeta_{p^2}^p=\sum_{k=0}^{p-1}\left[(-1)^k\frac{\zetax^k}{k!}\right]p^{\frac{k}{p-1}}+o(p) .$$
\end{proof}

\begin{lemma}\label{lem:tame_iso}
    Let $p\geq 3$ be a prime number. Then we have $\bbQ_p\left(\zeta_p\right)=\bbQ_p\left(\zetax p^{\frac{1}{p-1}}\right)$.
\end{lemma}
\begin{proof}
    Since $\bbQ_p\left(\zetax p^{\frac{1}{p-1}}\right)$ and $\bbQ_p\left(\zeta_p\right)$ have the same degree over $\bbQ_p$, we only need to show $\bbQ_p\left(\zetax p^{\frac{1}{p-1}}\right)\subseteq\bbQ_p\left(\zeta_p\right)$. It is enough to show $x^{p-1}=\frac{\left(\zeta_p-1\right)^{p-1}}{-p}$ has a solution in $\bbQ_p\left(\zeta_p\right)$.
    \par By \Cref{exp_first_rmk}, we have
    \begin{equation*}
        \frac{\left(\zeta_p-1\right)^{p-1}}{-p}=-p^{-1}\left(1-\zetax p^{\frac{1}{p-1}}+o\left(p^{\frac{1}{p-1}}\right)-1\right)^{p-1}=1+o\left(p^0\right) ;
    \end{equation*}
    thus we may set $\frac{\left(\zeta_p-1\right)^{p-1}}{-p}=1+M$, where $M$ is in the maximal ideal of $\bbQ_p\left(\zeta_p\right)$.
    Since $\binom{\frac{1}{p-1}}{k}\in\bbZ_p$, the binomial series
    $(1+M)^{\frac{1}{p-1}}=\sum_{k=0}^\infty \binom{\frac{1}{p-1}}{k}M^k$ converges in $\bbQ_p\left(\zeta_p\right)$.
\end{proof}

Since we apply the transfinite Newton algorithm for every $n\geq 2$ independently and the first $\aleph_0$ terms of the expansion is determined by $\frakz_{1,n}$, if we take the same $\zetax$ for every $\frakz_{1,n}=(-1)^n\zetax, n=1,3,\cdots$, the following result should be noticed:
\begin{corollary}\label{coro:54977}
    For every $n\geq 1$, we have $\aleph_0\left(\zeta_{p^n}\right)=\aleph_0\left(\zeta_{p^{n+1}}^p\right)$.
\end{corollary}
\begin{proof}
    Since $\zeta_{p^n}$ and $\zeta_{p^{n+1}}^p$ are both $p^n$-th primitive roots of unity, by \Cref{maintheorem1}, we only need to check that
    $$C_{\frac{1}{p^{n-1}(p-1)}}\left(\zeta_{p^{n+1}}^p\right)=\frakz_{1,n}=(-1)^n\zetax.$$
    One calculates
    \begin{align*}
        \zeta_{p^{n+1}}^p= & \left(1+(-1)^{n-1}\zetax p^{\frac{1}{p^n(p-1)}}+O\left((-1)^{n-1}\zetax p^{\frac{2}{p^n(p-1)}}\right)\right)^p                     \\
        =                  & 1+\left((-1)^{n-1}\zetax p^{\frac{1}{p^n(p-1)}}\right)^p+\left(O\left((-1)^{n-1}\zetax p^{\frac{2}{p^n(p-1)}}\right)\right)^p+O(p) \\
        =                  & 1+(-1)^n\zetax p^{\frac{1}{p^{n-1}(p-1)}}+O\left(p^{\frac{2}{p^{n-1}(p-1)}}\right),
    \end{align*}
    and consequently $C_{\frac{1}{p^{n-1}(p-1)}}\left(\zeta_{p^{n+1}}^p\right)=(-1)^n\zetax$, as expected.
\end{proof}

\subsection{Bell polynomials and Stirling numbers of the second kind}\label{Bellpoly}
In this paragraph, we introduce the notion of incomplete exponential Bell polynomials and Stirling numbers of the second kind, whose arithmetic properties will be used to estimate the $p$-adic valuation of  $\Lambda_{i,n}^{p^{n-1}}-1$ and $\Lambda_{i,n}^{p^n}-1$ in \Cref{Estimation}.
\subsubsection*{Generalities}
\par The Bell polynomials are used to study set partitions in combinatorial mathematics. Let $\alpha_l = (j_1, j_2, \cdots, j_l)\in \bbN^{l}$ be a multi-index. We denote its norm by $\vert\alpha_l\vert=j_1 + j_2 + \cdots + j_l$ and its factorial by
$\alpha_l! = \prod_{k=1}^l j_k!$. Let $\bm{x}=(x_1,\cdots, x_l)$ be a $l$-tuple of formal variables. The power of a multi-index $\alpha_l$ of $\bm{x}$ is defined by
$$ \bm{x}^{\alpha_l}:= \prod_{i=1}^lx_i^{j_i}.$$

\begin{definition}
    For integer numbers $n\geq k\geq 0$, the \textbf{incomplete exponential Bell polynomial} with parameter $(n,k)$ is  a polynomial given by
    \begin{gather*}
        \begin{aligned}
            B_{n,k}(x_{1},x_{2},\dots ,x_{n-k+1}):=\sum_{\substack{\alpha_{n-k+1}=(j_1,\cdots,j_{n-k+1})\in\bbN^{n-k+1} \\\vert\alpha_{n-k+1}\vert=k, \sum_{i=1}^{n-k+1}ij_i=n}} \frac{n!}{\alpha_{n-k+1}!}\left(\frac{x_1}{1!},\cdots, \frac{x_{n-k+1}}{(n-k+1)!}\right)^{\alpha_{n-k+1}}.
        \end{aligned}
    \end{gather*}

\end{definition}
With multinomial theorem, the incomplete exponential Bell polynomial can also be defined in terms of its generating function (cf. \cite[P.134 Theorem A]{Comtet1974}):
\begin{equation}\label{lem:bell_gen}
    \frac{1}{k!}\left(\sum_{m\geq 1}x_m\frac{t^m}{m!}\right)^k=\sum_{n\geq k}B_{n,k}(x_1,\cdots,x_{n-k+1})\frac{t^n}{n!},\ k=0,1,2,\cdots .
\end{equation}

From the algebraic point of view, the Bell polynomials can be computed using its generating function. In particular, if $k$ is small or close to $n$, the Bell polynomial $B_{n,k}(x_1,\cdots,x_{n-k+1})$ is easy to compute:
\begin{lemma}\label{lem:bell_calc}
    \begin{itemize}
        \item $B_{n,k}(x_1,\cdots,x_{n-k+1})=\begin{cases}x_n,& \text{ if } k=1; \\\frac{1}{2}\sum_{t=1}^{n-1}\binom{n}{t}x_t x_{n-t}, & \text{ if } k=2 .\end{cases}$
        \item$B_{n,k}(x_1,\cdots,x_{n-k+1})=\begin{cases}(x_1)^n, & \text{ if } k=n;\\ \binom{n}{2}(x_1)^{n-2}x_2 , &\text{ if } k=n-1; \\ \binom{n}{3}(x_1)^{n-3}x_3+3\binom{n}{4}(x_1)^{n-4}(x_2)^2, & \text{ if } k=n-2.\end{cases}
              $
    \end{itemize}
\end{lemma}

The special values of the incomplete exponential Bell polynomial at the points $(1,\cdots, 1)$ and $(\overbrace{1,\cdots,1}^{r},0,\cdots,0),$ called Stirling numbers of the second kind and $r$-restricted Stirling numbers of the second kind (cf. \cite{Komatsu2016,Mezo2014}) respectively. More precisely, we have the following definition.
\begin{definition}\leavevmode
    \begin{enumerate}
        \item For integer numbers $n\geq k\geq 0$, the \textbf{Stirling number of the second kind} is defined by
              $$\stirling{n}{k}= B_{n,k}(1,1,\cdots,1) ;$$
        \item For integer numbers $n\geq k\geq 0$ and positive integer $r$, the \textbf{$r$-restricted Stirling number of the second kind} is defined by
              $$\stirling{n}{k}_{\leq r}=
                  \begin{cases}
                      \stirling{n}{k},                                & \text{ if } n-k+1\leq r; \\
                      B_{n,k}(\overbrace{1,\cdots,1}^{r},0,\cdots,0), & \text{ otherwise.}
                  \end{cases}$$
    \end{enumerate}
\end{definition}

\par Using the generating function formula \Cref{lem:bell_gen} for Bell polynomials, one has:
\begin{lemma}[Generating function]\label{lem:stirling_gen}\leavevmode For $k\in \bbN$, then we have
    \begin{enumerate}
        \item\label{lem:stirling_gen:1} $$\frac{1}{k!}\left(\sum_{m\geq 1}\frac{t^m}{m!}\right)^k=\sum_{n\geq k}\stirling{n}{k}\frac{t^n}{n!} ;$$
        \item\label{lem:stirling_gen:2} $$\frac{1}{k!}\left(\sum_{m= 1}^r\frac{t^m}{m!}\right)^k=\sum_{n= k}^{rk}\stirling{n}{k}_{\leq r}\frac{t^n}{n!} .$$
    \end{enumerate}
\end{lemma}
By comparing \Cref{lem:bell_gen} and the second assertion of \Cref{lem:stirling_gen}, we have:
\begin{corollary}\label{finitesum} If $n\geq rk+1$, then we have
    $\stirling{n}{k}_{\leq r}=0$. Therefore, we can rewrite the second assertion of \Cref{lem:stirling_gen} as
    $$\frac{1}{k!}\left(\sum_{m= 1}^r\frac{t^m}{m!}\right)^k=\sum_{n= k}^\infty\stirling{n}{k}_{\leq r}\frac{t^n}{n!},\ k=0,1,2,\cdots .$$
\end{corollary}
We denote by $(x)_{n}=x(x-1)(x-2)\cdots (x-n+1)$ the falling factorials, which form a basis of the $\bbQ$-vector space $\bbQ[x]$.
The Stirling numbers of the second kind may also be characterized as the coordinate of powers of the indeterminate $x$ with respect to the basis consisting of the falling factorials (cf.  \cite[Page 207 Theorem B]{Comtet1974}) :  If $n>0$, one has
\begin{equation}\label{lem:stirling_falling}
    x^n=\sum_{m=0}^n\stirling{n}{m} (x)_m .
\end{equation}
\begin{corollary}\label{coro:stirling_factorial}
    \begin{equation*}
        \sum_{k=1}^n(-1)^{k-1}(k-1)!\stirling{n}{k}=
        \begin{cases}
            0, & n\geq 2; \\
            1, & n=1.
        \end{cases}
    \end{equation*}
\end{corollary}
\begin{proof}
    When $n=1$, the assertion follows from direct calculation. \par When $n\geq 2$, since $\binom{x}{k}=\binom{x-1}{k-1}\frac{x}{k}$ and $\stirling{n}{0}=0$, by \Cref{lem:stirling_falling} we know that
    $$\sum_{k=1}^n\stirling{n}{k}\binom{x-1}{k-1}(k-1)! =x^{n-1} .$$
    By setting $x=0$, we have
    $$\sum_{k=1}^n\stirling{n}{k}\binom{-1}{k-1}(k-1)! =0 ,$$
    where $\binom{-1}{k-1}=(-1)^{k-1}$.
\end{proof}
\subsubsection*{Arithmetic properties}
Now we establish several lemmas related to the arithmetic properties of (restricted) Stirling numbers of the second kind. The first lemma (cf. \Cref{babbage})  summarizes several well-known facts about the arithmetic properties of binomial coefficients, and the other lemmas (cf. \Cref{lem:stirling_arith1}, \Cref{lem:stirling_arith2}, \Cref{vrai} and \Cref{Jambon}) characterize the $\bmod\ p$ congruence properties of some special (restricted) Stirling numbers of the second kind, which will be used in \Cref{boringdog}, \Cref{preliminaire2}, \Cref{superluckydog} and \Cref{seconde}.

\begin{lemma}\label{babbage}Let $p\geq 3$ be a prime number and  $a,b\in\bbN$ be two natural numbers such that $a\geq b$. If $n$ is an integer satisfying $1\leq n\leq p-1$ and  $k$ is a positive integer, then we have
    \begin{enumerate}
        \item $v_p\left(\binom{p^n}{a}\right)=n-v_p(a)$;
        \item
              $\binom{pk}{n}\equiv pk\frac{(-1)^{n-1}}{n}\mod p^2$;

        \item $\binom{ap}{bp}\equiv \binom{a}{b}\mod p^2$.
    \end{enumerate}
\end{lemma}
\begin{proof}
    The first and the second assertions are well-known. The third assertion can be found in \cite[Theorem 1.6]{Grinberg2018}.
\end{proof}

\begin{lemma}\label{lem:stirling_arith1}
    Let $p$ be an odd prime number. For an integer $k$ that $1\leq k\leq p$, one has
    \begin{equation*}
        \stirling{p-1+k}{p}\equiv
        \begin{cases}
            1, & \text{ if } k=1 \text{ or } p; \\
            0, & \text{ otherwise}
        \end{cases}
        \pmod{p} .
    \end{equation*}
\end{lemma}
\begin{proof}
    By \cite[Theorem 5.2]{Stirling2010}, we have
    \begin{equation*}
        \stirling{n}{ap^m}\equiv
        \begin{cases}
            \binom{\frac{n-ap^{m-1}}{p-1}-1}{\frac{n-ap^m}{p-1}}, & \text{ if } n\equiv a\pmod{p-1}, \\
            0,                                                    & \text{ otherwise.}
        \end{cases}
        \pmod{p^m}
    \end{equation*}
    for positive integers $n,a,m$ that $m\geq 1$, $a>0$ and $n\geq ap^m$.
    The assertion follows by taking $n=p-1+k$ and $a=m=1$ in the above formula.
\end{proof}

\begin{lemma}\label{lem:stirling_arith2}
    Let $p$ be an odd prime number and $r$ an integer number satisfying $1\leq r<p-1$, then one has
    $$\stirling{r+p}{p}_{\leq r}=B_{r+p,p}(1,\cdots,1,0)\equiv 0 \mod{p} .$$
\end{lemma}
\begin{proof}
    If $r=1$, then $p+1\geq 1\cdot p+1$ and the result follows from \Cref{finitesum}.

    Now we suppose $r\geq 2$. By \cite[(1.3)]{Cvijovic2011}, one has the following identity:
    $$B_{n,k}(x_1,\cdots,x_{n-k+1})=\frac{1}{x_1}\cdot\frac{1}{n-k}\sum_{\alpha=1}^{n-k}\binom{n}{\alpha}\left((k+1)-\frac{n+1}{\alpha+1}\right)x_{\alpha+1}B_{n-\alpha,k}(x_1,\cdots,x_{n-\alpha-k+1}) .$$
    Let $n=r+p$, $k=p$ and
    $
        a_t=
        \begin{cases}
            1, & \text{ if } t\leq r; \\
            0, & \text{ if } t>r.
        \end{cases}
    $ Then, one has
    \begin{align*}
        \stirling{r+p}{p}_{\leq r}= & B_{r+p,p}(a_1,\cdots,a_{r+1})                                                                                                                     \\
        =                           & \frac{1}{r}\sum_{\alpha=1}^{r}\binom{r+p}{\alpha}\left((p+1)-\frac{r+p+1}{\alpha+1}\right)x_{\alpha+1}B_{r+p-\alpha,p}(a_1,\cdots,a_{r-\alpha+1}) \\
        =                           & \frac{1}{r}\sum_{\alpha=1}^{r-1}\binom{r+p}{\alpha}\left((p+1)-\frac{r+p+1}{\alpha+1}\right)\stirling{r+p-\alpha}{p} .
    \end{align*}
    Since $\alpha+1\leq r<p-1$ and $1<r-\alpha+1<p-1$, by \Cref{lem:stirling_arith1}, we have
    $$\stirling{r+p-\alpha}{p}\equiv 0 \pmod{p}.$$
    As a consequence, we have
    $$\stirling{r+p}{p}_{\leq r}\equiv 0 \pmod{p} .$$
\end{proof}

\begin{lemma}\label{vrai}Let $i$ be an integer that $1\leq i\leq p-1$ and $k\in \bbZ_{>0}$. Then for any integer $l\geq k$, we have
    $$v_p\left(\frac{k!}{l!}\stirling{l}{k}_{\leq i}\right)\geq 0.$$
\end{lemma}
\begin{proof}  For $j\in \bbN_{>0}$, we set $\delta_j=
        \begin{cases}
            1, & \text{ if } j\leq i; \\
            0, & \text{ otherwise.}
        \end{cases}$
    Recall that the incomplete exponential Bell polynomial is defined as follows:
    \begin{align*}
        B_{l,k}(x_{1},x_{2},\dots ,x_{l-k+1})=\sum_{\substack{\alpha_{l-k+1}=(j_1,\cdots,j_{l-k+1})\in\bbN^{l-k+1} \\\vert\alpha_{l-k+1}\vert=k, \sum_{t=1}^{l-k+1}tj_t=l}} \frac{l!}{\alpha_{l-k+1}!}\left(\frac{x_1}{1!},\cdots, \frac{x_{l-k+1}}{(l-k+1)!}\right)^{\alpha_{l-k+1}},
    \end{align*}
    and the $i$-restricted Stirling numbers of the second kind $\stirling{l}{k}_{\leq i}$ is the special value of $B_{l,k}$ at the point  $\delta=(\delta_j)_{1\leq j\leq l-k+1}$. For $\alpha=\left(j_1,\cdots,j_{l-k+1}\right)\in\bbN^{l-k+1}$, we set
    $$F_{l,k,i}(\alpha)=\binom{k}{j_1,\cdots,j_{l-k+1}}\left(\frac{\delta_1}{1!},\cdots, \frac{\delta_{l-k+1}}{(l-k+1)!}\right)^{\alpha} .$$
    Then we have
    $$\frac{k!}{l!}\stirling{l}{k}_{\leq i}=\sum_{\substack{\alpha=\left(j_1,\cdots,j_{l-k+1}\right)\in\bbN^{l-k+1}\\\vert\alpha\vert=k, \sum_{t=1}^{l-k+1}tj_t=l}}F_{l,k,i}(\alpha) ,$$
    and it enough to prove $v_p\left(F_{l,k,i}(\alpha)\right)\geq 0$ for all $\alpha$ in the above formula, which follows from the following discussions on the range of $i$:
    \begin{enumerate}
        \item  Suppose $l-k+1\leq i < p$. We have $v_p\left(\frac{\delta_m}{m!}\right)=v_p(\delta_m)=0$ for all $1\leq m\leq l-k+1$. Therefore,
              $$v_p\left(F_{l,k,i}(\alpha)\right)=v_p\left(\binom{k}{j_1,\cdots,j_{l-k+1}}\right)-\sum_{m=1}^{l-k+1}j_m v_p(m!)=v_p\left(\binom{k}{j_1,\cdots,j_{l-k+1}}\right) \geq 0 .$$
        \item Suppose $i<l-k+1$. For $\alpha=\left(j_1,\cdots,j_{l-k+1}\right)\in\bbN^{l-k+1}$, if there exists $m$ such that $i<m\leq l-k+1 $ and $j_m>0$, then $F_{l,k,i}(\alpha)=0$. If $j_{i+1}=\cdots=j_{l-k+1}=0$, then
              $$v_p\left(F_{l,k,i}(\alpha)\right)=v_p\left(\binom{k}{j_1,\cdots,j_{l-k+1}}\right)-\sum_{m=1}^i j_m v_p(m!)=v_p\left(\binom{k}{j_1,\cdots,j_{l-k+1}}\right) \geq 0 .$$

    \end{enumerate}
\end{proof}

\begin{lemma}\label{Jambon}
    For $n\in\bbN_{\geq 2}$, $1\leq s\leq p-1$ and $sp^{n-2}\leq t\leq p^{n-1}-1$, we have
    \begin{equation*}
        \frac{(sp^{n-2})!}{t!}\stirling{t}{sp^{n-2}}_{\leq p-1}\equiv \begin{cases}0 \mod p, & \text{ if }p^{n-2}\nmid t \\  \frac{s!}{(t/p^{n-2})!}\stirling{t/p^{n-2}}{s}  \mod p, & \text{ if } p^{n-2}\mid t       .\end{cases}
    \end{equation*}
\end{lemma}
\begin{proof}When $n=2$, the assertion follows from the fact $t-s+1\leq p-1$ and $\stirling{t}{s}_{\leq p-1}=\stirling{t}{s}$.

    Suppose $n\geq 3$. For $1\leq s\leq p-1$ and $sp^{n-2}\leq t$, we set $u_{s,t}= \min\{t-sp^{n-2}+1,p-1\}$. By the definition of restricted Stirling number of the second kind, we have
    $$\frac{(sp^{n-2})!}{t!}\stirling{t}{sp^{n-2}}_{\leq p-1}=\sum_{\substack{\alpha=(j_1,\cdots,j_{u_{s,t}})\in\bbN^{u_{s,t}}\\\vert\alpha\vert=sp^{n-2},\sum_{m=1}^{u_{s,t}}mj_m=t}}\binom{sp^{n-2}}{j_1,\cdots,j_{u_{s,t}}}\left(\frac{1}{1!},\cdots,\frac{1}{u_{s,t}!}\right)^\alpha .$$
    By separating this sum into two parts, we can write
    \begin{align*}
        \frac{(sp^{n-2})!}{t!}\stirling{t}{sp^{n-2}}_{\leq p-1}= & \sum_{\substack{\alpha=(j_1,\cdots,j_{u_{s,t}})\in \left(p^{n-2}\bbN\right)^{u_{s,t}}                        \\\vert\alpha\vert=sp^{n-2},\sum_{m=1}^{u_{s,t}}mj_m=t}}\binom{sp^{n-2}}{j_1,\cdots,j_{u_{s,t}}}\left(\frac{1}{1!},\cdots,\frac{1}{u_{s,t}!}\right)^\alpha\\
        +                                                        & \sum_{\substack{\alpha=(j_1,\cdots,j_{u_{s,t}})\in\bbN^{u_{s,t}}\backslash\left(p^{n-2}\bbN\right)^{u_{s,t}} \\\vert\alpha\vert=sp^{n-2},\sum_{m=1}^{u_{s,t}}mj_m=t}}\binom{sp^{n-2}}{j_1,\cdots,j_{u_{s,t}}}\left(\frac{1}{1!},\cdots,\frac{1}{u_{s,t}!}\right)^\alpha.
    \end{align*}
    If $\alpha=(j_1,\cdots,j_{u_{s,t}})\in\bbN^{u_{s,t}}\backslash\left(p^{n-2}\bbN\right)^{u_{s,t}}$, then, by the facts $\binom{sp^{n-2}}{j_m}$ is a factor of $\binom{sp^{n-2}}{j_1,\cdots,j_{u_{s,t}}}$ for all $1\leq m\leq u_{s,t}$ and $\binom{sp^{n-2}}{j_m}$ is divided by $p$ if $p^{n-2}\nmid j_m$, we have $\binom{sp^{n-2}}{j_1,\cdots,j_{u_{s,t}}}\left(\frac{1}{1!},\cdots,\frac{1}{u_{s,t}!}\right)^\alpha$ is divisible by $p$. Therefore, we have
    \begin{equation}\label{eq:35851}
        \frac{(sp^{n-2})!}{t!}\stirling{t}{sp^{n-2}}_{\leq p-1}=\sum_{\substack{\alpha=(j_1,\cdots,j_{u_{s,t}})\in \left(p^{n-2}\bbN\right)^{u_{s,t}}\\\vert\alpha\vert=sp^{n-2},\sum_{m=1}^{u_{s,t}}mj_m=t}}\binom{sp^{n-2}}{j_1,\cdots,j_{u_{s,t}}}\left(\frac{1}{1!},\cdots,\frac{1}{u_{s,t}!}\right)^\alpha+O(p) .
    \end{equation}
    By replacing $j_m$ with $\widehat{j}_m= j_m/p^{n-2}$ and replacing $\alpha$ with $\widehat{\alpha}= \left(\widehat{j}_1,\cdots,\widehat{j}_{u_{s,t}}\right)$, we can rewrite \Cref{eq:35851} as
    \begin{equation}\label{eq:30859}
        \begin{aligned}
              & \frac{(sp^{n-2})!}{t!}\stirling{t}{sp^{n-2}}_{\leq p-1}                                                   \\
            = & \sum_{\substack{\widehat{\alpha}=\left(\widehat{j}_1,\cdots,\widehat{j}_{u_{s,t}}\right)\in\bbN^{u_{s,t}} \\\vert \widehat{\alpha}\vert=s, \sum_{m=1}^{u_{s,t}}m\widehat{j}_m=t/p^{n-2}}}\binom{sp^{n-2}}{\widehat{j}_1 p^{n-2},\cdots,\widehat{j}_{u_{s,t}} p^{n-2}}\left(\left(\frac{1}{1!}\right)^{p^{n-2}},\cdots,\left(\frac{1}{u_{s,t}!}\right)^{p^{n-2}}\right)^{\widehat{\alpha}}+O(p) .
        \end{aligned}
    \end{equation}
    Notice that we have the identity
    \begin{gather*}
        \binom{sp^{n-2}}{\widehat{j}_1 p^{n-2},\cdots,\widehat{j}_{u_{s,t}} p^{n-2}}                                                                                                                                                                                   = \binom{\widehat{j}_1 p^{n-2}+\cdots+\widehat{j}_{u_{s,t}} p^{n-2}}{\widehat{j}_1 p^{n-2}}\binom{\widehat{j}_2 p^{n-2}+\cdots+\widehat{j}_{u_{s,t}} p^{n-2}}{\widehat{j}_2 p^{n-2}}\cdots\binom{\widehat{j}_{u_{s,t}} p^{n-2}}{\widehat{j}_{u_{s,t}} p^{n-2}} ,
    \end{gather*}
    and by applying the formula $\binom{ap}{bp}\equiv \binom{a}{b}\pmod{p^2}$ (cf. \Cref{babbage}) to this identity, we obtain
    \begin{align*}
        \binom{sp^{n-2}}{\widehat{j}_1 p^{n-2},\cdots,\widehat{j}_{u_{s,t}} p^{n-2}}
        = & \binom{\widehat{j}_1 +\cdots+\widehat{j}_{u_{s,t}} }{\widehat{j}_1 }\binom{\widehat{j}_2 +\cdots+\widehat{j}_{u_{s,t}} }{\widehat{j}_2 }\cdots\binom{\widehat{j}_{u_{s,t}} }{\widehat{j}_{u_{s,t}}}+O(p) \\
        = & \binom{s}{\widehat{j}_1,\cdots,\widehat{j}_{u_{s,t}}}+O(p) .
    \end{align*}
    Additionally, for all $m\in\{1,\cdots,u_{s,t}\}$, we have
    $$\left(\frac{1}{m!}\right)^{\widehat{j}_m p^{n-2}}=\left(\frac{1}{m!}\right)^{\widehat{j}_m}+O(p).$$
    Therefore, we can rewrite \Cref{eq:30859} as
    \begin{equation}
        \frac{(sp^{n-2})!}{t!}\stirling{t}{sp^{n-2}}_{\leq p-1}=\sum_{\substack{\widehat{\alpha}=\left(\widehat{j}_1,\cdots,\widehat{j}_{u_{s,t}}\right)\in\bbN^{u_{s,t}}\\\vert \widehat{\alpha}\vert=s, \sum_{m=1}^{u_{s,t}}m\widehat{j}_m=t/p^{n-2}}}\binom{s}{\widehat{j}_1,\cdots,\widehat{j}_{u_{s,t}}}\left(\frac{1}{1!},\cdots,\frac{1}{u_{s,t}!}\right)^{\widehat{\alpha}}+O(p) .
    \end{equation}
    If $p^{n-2}\nmid t$, then the summation above is void and consequently $v_p\left(\frac{(sp^{n-2})!}{t!}\stirling{t}{sp^{n-2}}_{\leq p-1}\right)\geq 1$.
    \par It remains to deal with the case $p^{n-2}\mid t$. By setting $t=\widehat{t}p^{n-2}$ with $s\leq\widehat{t}\leq p-1$, we have
    \begin{equation*}
        \frac{(sp^{n-2})!}{t!}\stirling{t}{sp^{n-2}}_{\leq p-1}=\frac{(sp^{n-2})!}{(\widehat{t}p^{n-2})!}\stirling{\widehat{t}p^{n-2}}{sp^{n-2}}_{\leq p-1}.
    \end{equation*}
    We conclude our assertion in this case by the following discussion on the relation between $\widehat{t}$ and $s$.
    \begin{enumerate}
        \item If $\widehat{t}=s$, we have $$\frac{(sp^{n-2})!}{(\widehat{t}p^{n-2})!}\stirling{\widehat{t}p^{n-2}}{sp^{n-2}}_{\leq p-1}=1=\frac{s!}{\widehat{t}!}\stirling{\widehat{t}}{s} .$$
        \item If $\widehat{t}>s$, we have $\widehat{t}p^{n-2}-sp^{n-2}+1\geq p^{n-2}+1> p-1$. Therefore, $u_{s,\widehat{t}p^{n-2}}=p-1$ and
              \begin{align*}
                  \frac{(sp^{n-2})!}{(\widehat{t}p^{n-2})!}\stirling{\widehat{t}p^{n-2}}{sp^{n-2}}_{\leq p-1}
                  =  \sum_{\substack{\widehat{\alpha}=\left(\widehat{j}_1,\cdots,\widehat{j}_{p-1}\right)\in\bbN^{p-1} \\\vert \widehat{\alpha}\vert=s, \sum_{m=1}^{p-1}m\widehat{j}_m=\widehat{t}}}\binom{s}{\widehat{j}_1,\cdots,\widehat{j}_{p-1}}\left(\frac{1}{1!},\cdots,\frac{1}{(p-1)!}\right)^{\widehat{\alpha}}+O(p) .
              \end{align*}
              If there exists $p-1\geq r>\widehat{t}-s+1$ such that $j_r\neq 0$, then
              $$1\widehat{j}_1+\cdots+(p-1)\widehat{j}_{p-1}\geq 1\cdot(s-1)+r>\widehat{t} ,$$
              which contradicts to the condition that $\sum_{m=1}^{p-1}m\widehat{j}_m=\widehat{t}$.
              Therefore, $\widehat{j}_r=0$ for all $r>\widehat{t}-s+1$. As a  consequence,
              \begin{align*}
                    & \frac{(sp^{n-2})!}{(\widehat{t}p^{n-2})!}\stirling{\widehat{t}p^{n-2}}{sp^{n-2}}_{\leq p-1}                               \\
                  = & \sum_{\substack{\widehat{\alpha}=\left(\widehat{j}_1,\cdots,\widehat{j}_{\widehat{t}-s+1}\right)\in\bbN^{\widehat{t}-s+1} \\\vert \widehat{\alpha}\vert=s, \sum_{m=1}^{\widehat{t}-s+1}m\widehat{j}_m=\widehat{t}}}\binom{s}{\widehat{j}_1,\cdots,\widehat{j}_{\widehat{t}-s+1}}\left(\frac{1}{1!},\cdots,\frac{1}{(\widehat{t}-s+1)!}\right)^{\widehat{\alpha}}+O(p)\\
                  = & \frac{s!}{\widehat{t}!}\stirling{\widehat{t}}{s}+O(p) .
              \end{align*}
    \end{enumerate}

\end{proof}

\subsubsection*{The main technical propositions}
In this paragraph, we establish our main technical propositions (cf. \Cref{superluckydog}, \Cref{boringdog} and \Cref{preliminaire2}) using the arithmetic properties of (restricted) Stirling numbers of the second kind.

\begin{proposition}\label{superluckydog}For $n\in \bbN_{\geq 2}$, we have
    $$  \left(\sum_{l=0}^{p-1}\frac{(-1)^{ln}}{l!}\zetax^l p^{\frac{l}{p^{n-1}(p-1)}}\right)^{p^{n-1}}-1
        =                \sum_{l=1}^{p-1}\frac{(-1)^l}{[l!]}\zetax^l p^{\frac{l}{p-1}}+\zetax p^{1+\frac{1}{p(p-1)}}+O\left(p^{1+\frac{1}{p-1}}\right).$$
\end{proposition}
\begin{proof}
    Let $\lambda_n=(-1)^n\zetax p^{\frac{1}{p^{n-1}(p-1)}}$, and we rewrite left-hand side of the equality as
    \begin{equation}\label{eq:47220}
        \left(\sum_{l=0}^{p-1}\frac{(-1)^{ln}}{l!}\zetax^l p^{\frac{l}{p^{n-1}(p-1)}}\right)^{p^{n-1}}-1= \left(\sum_{l=1}^{p-1}\frac{\lambda_n^l}{l!}\right)^{p^{n-1}}+ H(n),
    \end{equation}
    where
    $H(n)= \sum_{j=1}^{p^{n-1}-1}\binom{p^{n-1}}{j}\left(\sum_{l=1}^{p-1}\frac{\lambda_n^l}{l!}\right)^j .$

    Note that $v_p\left(\binom{p^{n-1}}{j}\left(\sum_{l=1}^{p-1}\frac{\lambda_n^l}{l!}\right)^j\right)=n-1-v_p(j)+\frac{j}{p^{n-1}(p-1)}$
    and the condition $$n-1-v_p(j)+\frac{j}{p^{n-1}(p-1)}<1+\frac{1}{p-1}$$ implies $v_p(j)=n-2$.
    We can rewrite $H(n)$ as
    \begin{equation}
        \sum_{{s}=1}^{p-1}\binom{p^{n-1}}{{s}p^{n-2}}\left(\sum_{l=1}^{p-1}\frac{\lambda_n^l}{l!}\right)^{{s}p^{n-2}}+O\left(p^{1+\frac{1}{p-1}}\right) .
    \end{equation}
    Using \Cref{babbage}, one can further simplify it as
    $$H(n)=\sum_{s=1}^{p-1}\binom{p}{s}\left(\sum_{l=1}^{p-1}\frac{\lambda_n^l}{l!}\right)^{sp^{n-2}}+O\left(p^{1+\frac{1}{p-1}}\right) .$$

    Applying the generating function formula for restricted Stirling numbers of the second kind, we obtain
    \begin{align*}
        H(n)= & \sum_{s=1}^{p-1}\binom{p}{s}(sp^{n-2})!\sum_{t=sp^{n-2}}^\infty \stirling{t}{sp^{n-2}}_{\leq p-1}\frac{\lambda_n^t}{t!}+O\left(p^{1+\frac{1}{p-1}}\right)               \\
        =     & \sum_{s=1}^{p-1}\binom{p}{s}\sum_{t=sp^{n-2}}^\infty \left(\frac{(sp^{n-2})!}{t!}\stirling{t}{sp^{n-2}}_{\leq p-1}\right)\lambda_n^t+O\left(p^{1+\frac{1}{p-1}}\right).
    \end{align*}
    By \Cref{vrai}, we know that $\frac{(sp^{n-2})!}{t!}\stirling{t}{sp^{n-2}}_{\leq p-1}$ has non-negative valuation. Note that we have $v_p(\lambda_n)=\frac{1}{p^{n-1}(p-1)}$ and
    for $t\geq p^{n-1}$, we have $v_p\left(\lambda_n^t\right)\geq\frac{1}{p-1}$. Thus, we can assemble the terms with $t\geq p^{n-1}$ of $H(n)$ into the error term:
    \begin{equation}\label{eq:36550}
        H(n)=\sum_{s=1}^{p-1}\binom{p}{s}\sum_{t=sp^{n-2}}^{p^{n-1}-1} \left(\frac{(sp^{n-2})!}{t!}\stirling{t}{sp^{n-2}}_{\leq p-1}\right)\lambda_n^t+O\left(p^{1+\frac{1}{p-1}}\right).
    \end{equation}
    We denote by $\widehat{t}=\frac{t}{p^{n-2}}$.  By \Cref{Jambon}, we obtain
    $$H(n) = \sum_{s=1}^{p-1}\binom{p}{s}\sum_{\widehat{t}=s}^{p-1} \left(\frac{s!}{\widehat{t}!}\stirling{\widehat{t}}{s}\right)\lambda_n^{\widehat{t}p^{n-2}}+O\left(p^{1+\frac{1}{p-1}}\right) .$$
    By exchanging the order of the summations and using the second assertion of \Cref{babbage}, we have
    \begin{equation}\label{eq:22082}
        \begin{split}
            H(n)  =  & p\sum_{\widehat{t}=1}^{p-1}\frac{\lambda_n^{\widehat{t}p^{n-2}}}{\widehat{t}!}\sum_{s=1}^{\widehat{t}} (-1)^{s-1}(s-1)!\stirling{\widehat{t}}{s}+O\left(p^{1+\frac{1}{p-1}}\right) \\
            =&p \lambda_n^{p^{n-2}}+O\left(p^{1+\frac{1}{p-1}}\right),
        \end{split}
    \end{equation}
    where the last equality follows from  \Cref{coro:stirling_factorial}.

    For the term $\left(\sum_{l=1}^{p-1}\frac{\lambda_n^l}{l!}\right)^{p^{n-1}}$, by multinomial theorem, one has
    $$\left(\sum_{l=1}^{p-1}\frac{\lambda_n^l}{l!}\right)^{p^{n-1}}=\sum_{\substack{j_1,\cdots,j_{p-1}\in\bbN\\j_1+\cdots+j_{p-1}=p^{n-1}}}\binom{p^{n-1}}{j_1,\cdots,j_{p-1}}\prod_{l=1}^{p-1}\left(\frac{\lambda_n^l}{l!}\right)^{j_l}.$$
    If $j_1,\cdots,j_{p-1}<p^{n-1}$, then we have
    \begin{align*}
        v_p\left(\binom{p^{n-1}}{j_1,\cdots,j_{p-1}}\prod_{l=1}^{p-1}\left(\frac{\lambda_n^l}{l!}\right)^{j_l}\right)= & v_p\left(\binom{p^{n-1}}{j_1,\cdots,j_{p-1}}\right)+\sum_{l=1}^{p-1}lj_l v_p(\lambda_n) \\
        \geq                                                                                                           & 1+\frac{1}{p^{n-1}(p-1)}\sum_{l=1}^{p-1}1\cdot j_l                                      \\
        =                                                                                                              & 1+\frac{1}{p-1}.
    \end{align*}
    If there exists a $l\in\{1,\cdots,p-1\}$ such that $j_l=p^{n-1}$, one calculates
    \begin{align*}
        \left(\frac{\lambda_n^l}{l!}\right)^{p^{n-1}}= & (-1)^l\zetax^l p^{\frac{l}{p-1}}\frac{1}{(l!)^{p^{n-1}}}
        =                                    (-1)^l\zetax^l p^{\frac{l}{p-1}}\left(\frac{1}{[l!]}+O(p)\right)                             \\
        =                                              & \frac{(-1)^l}{[l!]}\zetax^l p^{\frac{l}{p-1}}+O\left(p^{1+\frac{1}{p-1}}\right).
    \end{align*}
    In conclusion, we have
    \begin{align*}
        \left(\sum_{l=1}^{p-1}\frac{\lambda_n^l}{l!}\right)^{p^{n-1}}= & \sum_{l=1}^{p-1}\left(\frac{\lambda_n^l}{l!}\right)^{p^{n-1}}+O\left(p^{1+\frac{1}{p-1}}\right)   \\
        =                                                              & \sum_{l=1}^{p-1}\frac{(-1)^l}{[l!]}\zetax^l p^{\frac{l}{p-1}}+O\left(p^{1+\frac{1}{p-1}}\right) .
    \end{align*}
    Combining with \Cref{eq:47220,eq:22082}, we have
    \begin{align*}
          & \left(\sum_{l=0}^{p-1}\frac{(-1)^{ln}}{l!}\zetax^l p^{\frac{l}{p^{n-1}(p-1)}}\right)^{p^{n-1}}-1                                \\
        = & \left(\sum_{l=1}^{p-1}\frac{\lambda_n^l}{l!}\right)^{p^{n-1}}+ H(n)                                                             \\
        = & \sum_{l=1}^{p-1}\frac{(-1)^l}{[l!]}\zetax^l p^{\frac{l}{p-1}}+\zetax p^{1+\frac{1}{p(p-1)}}+O\left(p^{1+\frac{1}{p-1}}\right) ,
    \end{align*}
    as expected.
\end{proof}

\begin{proposition}\label{boringdog}
    For $n\in\bbN_{\geq 2}$, we have
    $$\left(\sum_{l=0}^{p-1}\frac{(-1)^l}{l!}\zetax^l p^{\frac{l}{p-1}}\right)^p-1=O\left(p^{2+\frac{1}{p-1}}\right) .$$
\end{proposition}
\begin{proof}
    Let $\theta_n=-\zetax p^{\frac{1}{p-1}}$, then by the generate function of the restricted Stirling number of the second kind we have
    \begin{equation}\label{eq:51692}
        \left(\sum_{l=0}^{p-1}\frac{\theta_n^l}{l!}\right)^p-1=\sum_{j=1}^p\binom{p}{j}\left(\sum_{l=1}^{p-1}\frac{\theta_n^l}{l!}\right)^j=\sum_{j=1}^p\binom{p}{j}\sum_{k=j}^\infty \frac{j!}{k!}\stirling{k}{j}_{\leq p-1}\theta_n^k.
    \end{equation}
    Notice that
    $$v_p\left(\binom{p}{j}\left(\frac{j!}{k!}\stirling{k}{j}_{\leq p-1}\right)\theta_n^k\right)\geq 1-v_p(j)+\frac{k}{p-1} ,$$
    by assembling terms with valuation equal or greater than $2+\frac{1}{p-1}$, we can rewrite \Cref{eq:51692} as
    \begin{align*}
        \left(\sum_{l=0}^{p-1}\frac{\theta_n^l}{l!}\right)^p-1=\sum_{j=1}^{p-1}\binom{p}{j}\sum_{k=j}^{p-1} \frac{j!}{k!}\stirling{k}{j}_{\leq p-1}\theta_n^k+\sum_{k=p}^{2p-2}\frac{p!}{k!}\stirling{k}{p}_{\leq p-1}\theta_n^k+O\left(p^{2+\frac{1}{p-1}}\right) .
    \end{align*}
    By the definition of restricted Stirling number of the second kind and changing the order of summations, we can further reduce this to
    \begin{align*}
        \left(\sum_{l=0}^{p-1}\frac{\theta_n^l}{l!}\right)^p-1= & \sum_{k=1}^{p-1}\frac{\theta_n^k}{k!}\sum_{j=1}^k\binom{p}{j}j!\stirling{k}{j}+\sum_{k=p}^{2p-2}\frac{p!}{k!}\stirling{k}{p}\theta_n^k+O\left(p^{2+\frac{1}{p-1}}\right)      \\
        =                                                       & p\sum_{k=1}^{p-1}\frac{\theta_n^k}{k!}\sum_{j=1}^k(-1)^{j-1}(j-1)!\stirling{k}{j}+\sum_{k=p}^{2p-2}\frac{p!}{k!}\stirling{k}{p}\theta_n^k +O\left(p^{2+\frac{1}{p-1}}\right).
    \end{align*}
    By \Cref{coro:stirling_factorial},
    $$p\sum_{k=1}^{p-1}\frac{\theta_n^k}{k!}\sum_{j=1}^k(-1)^{j-1}(j-1)!\stirling{k}{j}=p\theta_n .$$
    On the other hand, since $v_p(p!)=v_p(k!)=1$ for $k=p,\cdots 2p-2$, by \Cref{lem:stirling_arith1} we have
    $$\frac{p!}{k!}\stirling{k}{p}=
        \begin{cases}
            O(p),   & \text{ if }p<k\leq 2p-2 ; \\
            1+O(p), & \text{ if }k=p ,
        \end{cases}
    $$
    and consequently
    $$\left(\sum_{l=0}^{p-1}\frac{\theta_n^l}{l!}\right)^p-1=p\theta_n+\theta_n^p+O\left(p^{2+\frac{1}{p-1}}\right)=O\left(p^{2+\frac{1}{p-1}}\right) .$$
\end{proof}

\begin{proposition}\label{preliminaire2}Let $p$ be a prime and let $1\leq i<p-1$ be an integer. For $1\leq l\leq i+1$ an integer, we set $$G_i(l)=\left(\sum_{k=1}^l(-1)^{k-1}(k-1)!\stirling{l}{k}_{\leq i}\right)+\frac{p \cdot l!}{(l+p-1)!}\stirling{l+p-1}{p}_{\leq i} .$$
    Then we have $G_i(l)=\begin{cases}-1+O(p), &\text{ if } l=i+1; \\ O(p), & \text{ if } l\leq i .\end{cases}$
\end{proposition}
\begin{proof} We rewrite $G_i(l)$ as following: $$G_i(l) =\begin{cases}\sum_{k=1}^l(-1)^{k-1}(k-1)!\stirling{l}{k}+\stirling{p-1+l}{p}\frac{ p\cdot l!}{(l+p-1)!}, & \text{if }l\leq i; \\  \sum_{k=1}^{i+1}(-1)^{k-1}(k-1)!\stirling{i+1}{k}_{\leq i}+\frac{p \cdot (i+1)!}{(i+p)!}\stirling{i+p}{p}_{\leq i},& \text{if }l= i+1.\end{cases}$$
    Recall that, the \Cref{coro:stirling_factorial} says
    \begin{equation*}
        \sum_{k=1}^n(-1)^{k-1}(k-1)!\stirling{n}{k}=
        \begin{cases}
            0, & n\geq 2; \\
            1, & n=1.
        \end{cases}
    \end{equation*}

    \begin{itemize}[leftmargin=*]
        \item Suppose $l\leq i$.                     If $l=1$, then one has
              $$G_i(1)=1+\stirling{p}{p}\frac{1}{(p-1)!}\equiv 0 \mod p.$$
              If $1<l\leq i<p-1$, by \Cref{lem:stirling_arith1} and \Cref{coro:stirling_factorial}, one has
              $$G_i(n)=0+\stirling{p-1+l}{p}\frac{p\cdot l!}{(l+p-1)!}\equiv 0 \mod p.$$

        \item Suppose $l=i+1$, by \Cref{lem:stirling_arith2} and \Cref{coro:stirling_factorial}, one has
              $$G_i(i+1)= \sum_{k=1}^{i+1}(-1)^{k-1}(k-1)!\stirling{i+1}{k}_{\leq i}+\frac{p \cdot (i+1)!}{(i+p)!}\stirling{i+p}{p}_{\leq i}.$$
              For $2\leq k\leq i+1$, one has $\stirling{i+1}{k}_{\leq i}=\stirling{i+1}{k}$, therefore
              \begin{align*}
                  G_i(i+1)= & (-1)^{1-1}(1-1)!\stirling{i+1}{1}_{\leq i} +\sum_{k=2}^{i+1}(-1)^{k-1}(k-1)!\stirling{i+1}{k}  +\stirling{i+p}{p}_{\leq i}\frac{p\cdot (i+1)!}{(i+p)!} \\
                  =         & 0-(-1)^{1-1}(1-1)!\stirling{i+1}{1}+\sum_{k=1}^{i+1}(-1)^{k-1}(k-1)!\stirling{i+1}{k}+O(p)\frac{p(i+1)!}{(i+p)!}                                       \\
                  =         & -1+0+O(p)                                                                                                                                              \\
                  =         & -1+O(p) .
              \end{align*}
    \end{itemize}
\end{proof}

\subsection{Estimation of $\Lambda_{i,n}^{p^{n-1}}-1$ and $\Lambda_{i,n}^{p^n}-1$}\label{Estimation}
Let $n\in\bbN_{\geq 2}$. Recall that we set
$$\Lambda_{i,n}=
    \begin{cases}
        \sum_{k=0}^i \frac{(-1)^{kn}}{[k!]}\zetax^k p^{\frac{k}{p^{n-1}(p-1)}},                   & \text{ for } 0\leq i\leq p-1, \\
        \Lambda_{p-1,n}+ \sum_{l=n}^{i-p+n}(-1)^n\zetax p^{\frac{1}{p^{n-2}(p-1)}-\frac{1}{p^l}}, & \text{ for }  i\geq p.
    \end{cases}$$
As indicated in \Cref{statement}, for $i\in\bbN_{>0}$ and $0\leq k\leq p^{n-1}$, we can describe the coefficients $b_{p^{n-1}(p-1)-k}^{(i,n)}$ of the $i$-approximation polynomial $\Phi^{(i,n)}$ by the following formula
$$b_{p^{n-1}(p-1)-k}^{(i,n)}=\begin{cases}\frac{\Lambda_{i-1,n}^{p^n}-1}{\Lambda_{i-1,n}^{p^{n-1}}-1},                                                                                                                                                            & \text{ if } k=0;                 \\
        \frac{(-1)^{k-1}p^{n-1}}{k\Lambda_{i-1,n}^k}\left(\frac{p\Lambda_{i-1,n}^{p^n}}{\Lambda_{i-1,n}^{p^{n-1}}-1}-\Lambda_{i-1,n}^{p^{n-1}}\frac{\Lambda_{i-1,n}^{p^{n}}-1}{(\Lambda_{i-1,n}^{p^{n-1}}-1)^2}\right)+O(p^n) , & \text{ if } 1\leq k\leq p^{n-1}.\end{cases}$$
This leads us to estimate the $p$-adic valuation of $\Lambda_{i,n}^{p^{n-1}}-1$ and $\Lambda_{i,n}^{p^{n}}-1$ in \Cref{premier} and \Cref{seconde} respectively. In general, we obtain the estimation by induction, but since the formula for $\Lambda_{i,n}$ in the ranges $1\leq i<p-1$ and $ p-1\leq i$ are different, the statements will be separated into two parts.
\begin{proposition}\label{premier}
    Let $n\in \bbN_{\geq 2}$.
    \begin{enumerate}
        \item If $1\leq i< p-1$, we have
              $$\Lambda_{i,n}^{p^{n-1}}-1=\sum_{l=1}^{i}\frac{(-1)^l}{[l!]}\zetax^l p^{\frac{l}{p-1}}+O\left(p^{1+\frac{1}{p(p-1)}}\right).$$

        \item If $p-1\leq i$, we have
              $$\Lambda_{i,n}^{p^{n-1}}-1=\sum_{l=1}^{p-1} \frac{(-1)^l}{[l!]}\zetax^l p^{\frac{l}{p-1}}+\zetax p^{1+\frac{1}{p-1}-\frac{1}{p^{i-p+2}}}+O\left(p^{1+\frac{1}{p-1}}\right).$$
    \end{enumerate}
\end{proposition}
\begin{proof}We prove this lemma by induction on $i$.
    \begin{enumerate}
        \item If $i=1$, then we have
              \begin{align*}
                  \Lambda_{1,n}^{p^{n-1}}-1= & \left(1+(-1)^n\zetax p^{\frac{1}{p^{n-1}(p-1)}}\right)^{p^{n-1}}  -1                                                                                                                  =  \sum_{k=1}^{p^{n-1}}\binom{p^{n-1}}{k}(-1)^{kn}\zetax^k p^{\frac{k}{p^{n-1}(p-1)}} \\
                  =                          & -\zetax p^{\frac{1}{p-1}}+O\left(p^{1+\frac{1}{p(p-1)}}\right) .
              \end{align*}
              Suppose the lemma is true for $j$ with $1\leq j\leq  i-1\leq p-3$. Then, we have
              \begin{gather*}
                  \begin{aligned}
                      \Lambda_{i,n}^{p^{n-1}}-1= & \left(\Lambda_{i-1,n}+\frac{(-1)^{in}}{[i!]}\zetax^i p^{\frac{i}{p^{n-1}(p-1)}}\right)^{p^{n-1}}-1                                                                                                                                         \\
                      =                          & \Lambda_{i-1,n}^{p^{n-1}}-1+\sum_{k=1}^{p^{n-1}-1}\binom{p^{n-1}}{k}\Lambda_{i-1,n}^{p^{n-1}-k}\frac{(-1)^{ikn}}{[i!]^k}\zetax^{ik} p^{\frac{ik}{p^{n-1}(p-1)}}+\frac{(-1)^{inp^{n-1}}}{[i!]^{p^{n-1}}}\zetax^{ip^{n-1}} p^{\frac{i}{p-1}} \\
                      =                          & \Lambda_{i-1,n}^{p^{n-1}}-1+\frac{(-1)^i}{[i!]}\zetax^{i} p^{\frac{i}{p-1}}+O\left(p^{1+\frac{i}{p-1}}\right).
                  \end{aligned}
              \end{gather*}
              Therefore, the induction hypothesis allows us to conclude this case.

        \item  If $i=p-1$, then we have
              \begin{equation}
                  \begin{split}
                      \Lambda_{p-1,n}^{p^{n-1}}-1= & \left(\sum_{l=0}^{p-1}\frac{(-1)^{ln}}{[l!]}\zetax^l p^{\frac{l}{p^{n-1}(p-1)}}\right)^{p^{n-1}}-1                                                                              \\
                      =                  & \left(\sum_{l=0}^{p-1}\frac{(-1)^{ln}}{l!}\zetax^l p^{\frac{l}{p^{n-1}(p-1)}}+O\left(p^{1+\frac{2}{p^{n-1}(p-1)}}\right)\right)^{p^{n-1}}-1                                           \\
                      =                  & \left(\sum_{j=0}^{p^{n-1}}\binom{p^{n-1}}{j}\left(\sum_{l=0}^{p-1}\frac{(-1)^{ln}}{l!}\zetax^l p^{\frac{l}{p^{n-1}(p-1)}}\right)^{p^{n-1}-j}\left(O\left(p^{1+\frac{2}{p^{n-1}(p-1)}}\right)\right)^j\right)-1.
                      \label{eq:19663}
                  \end{split}
              \end{equation}
              For $1\leq j\leq p^{n-1}$, we observe that
              \begin{gather*}
                  \begin{aligned}
                      \binom{p^{n-1}}{j}\left(\sum_{l=0}^{p-1}\frac{(-1)^{ln}}{l!}\zetax^l p^{\frac{l}{p^{n-1}(p-1)}}\right)^{p^{n-1}-j}\left(O\left(p^{1+\frac{2}{p^{n-1}(p-1)}}\right)\right)^j=O\left(p^{n-1-v_p(j)+j+\frac{2j}{p^{n-1}(p-1)}}\right) .
                  \end{aligned}
              \end{gather*}
              Since $v_p(j)\leq n-1$ and $j\geq 1$, we know that $$n-1-v_p(j)+j+\frac{2j}{p^{n-1}(p-1)}>2 ,$$
              and thus \Cref{eq:19663} can be written as
              \begin{align*}
                  \Lambda_{p-1,n}^{p^{n-1}}-1= & \left(\sum_{l=0}^{p-1}\frac{(-1)^{ln}}{l!}\zetax^l p^{\frac{l}{p^{n-1}(p-1)}}\right)^{p^{n-1}}-1+\sum_{j=1}^{p^{n-1}}O\left(p^2\right) \\
                  =                            & \left(\sum_{l=0}^{p-1}\frac{(-1)^{ln}}{l!}\zetax^l p^{\frac{l}{p^{n-1}(p-1)}}\right)^{p^{n-1}}-1+O\left(p^2\right) .
              \end{align*}

              By \Cref{superluckydog}, we have
              \begin{gather*}
                  \left(\sum_{l=0}^{p-1}\frac{(-1)^{ln}}{l!}\zetax^l p^{\frac{l}{p^{n-1}(p-1)}}\right)^{p^{n-1}}-1=\sum_{l=1}^{p-1}\frac{(-1)^l}{[l!]}\zetax^lp^{\frac{l}{p-1}}+ \zetax p^{1+\frac{1}{p(p-1)}}+O\left(p^{1+\frac{1}{p-1}}\right).
              \end{gather*}
              As a consequence, we obtain
              $$\Lambda_{p-1,n}^{p^{n-1}}-1=\sum_{l=1}^{p-1}\frac{(-1)^l}{[l!]}\zetax^l p^{\frac{l}{p-1}}+ \zetax p^{1+\frac{1}{p-1}-\frac{1}{p}}+O\left(p^{1+\frac{1}{p-1}}\right). $$

        \item
              Now we suppose the formula holds for all $j$ with $p-1\leq j\leq i-1$, i.e.
              $$\Lambda_{j,n}^{p^{n-1}}-1=\sum_{l=1}^{p-1} \frac{(-1)^l}{[l!]}\zetax^l p^{\frac{l}{p-1}}+\zetax p^{1+\frac{1}{p-1}-\frac{1}{p^{j-p+2}}}+O\left(p^{1+\frac{1}{p-1}}\right).$$

              One has
              \begin{equation}\label{luckydog}
                  \begin{split}
                      \Lambda_{i,n}^{p^{n-1}}-1= & \left(\Lambda_{i-1,n}+(-1)^n\zetax p^{\frac{1}{p^{n-2}(p-1)}-\frac{1}{p^{n-p+i}}}\right)^{p^{n-1}}-1                                                                                                                      \\
                      =                    & \Lambda_{i-1,n}^{p^{n-1}}-1+\left((-1)^n\zetax p^{\frac{1}{p^{n-2}(p-1)}-\frac{1}{p^{n-p+i}}}\right)^{p^{n-1}}\\
                      &+\sum_{k=1}^{p^{n-1}-1}\binom{p^{n-1}}{k}\Lambda_{i-1,n}^{p^{n-1}-k}\left((-1)^n\zetax p^{\frac{1}{p^{n-2}(p-1)}-\frac{1}{p^{n-p+i}}}\right)^k.
                  \end{split}
              \end{equation}
              Notice that for every $k\in\{1,\cdots,p^{n-1}-1\}$,
              \begin{gather*}
                  v_p\left(\binom{p^{n-1}}{k}\Lambda_{i-1,n}^{p^{n-1}-k}\left((-1)^n\zetax p^{\frac{1}{p^{n-2}(p-1)}-\frac{1}{p^{n-p+i}}}\right)^k\right)=(n-1)-v_p(k)+\frac{k}{p^{n-2}}\left(\frac{1}{p-1}-\frac{1}{p^{i-p+2}} \right).
              \end{gather*}
              Thus,  the condition with variable $k$
              $$v_p\left(\binom{p^{n-1}}{k}\Lambda_{i-1,n}^{p^{n-1}-k}\left((-1)^n\zetax p^{\frac{1}{p^{n-2}(p-1)}-\frac{1}{p^{n-p+i}}}\right)^k\right)< 1+\frac{1}{p-1}$$
              implies $k=p^{n-2}$.  Since $\Lambda_{i-1,n}=1+O\left(p^{\frac{1}{p^{n-1}(p-1)}}\right)$, we have
              \begin{equation}\label{eq:1060}
                  \begin{split}
                      & \binom{p^{n-1}}{p^{n-2}}\left((-1)^n\zetax p^{\frac{1}{p^{n-2}(p-1)}-\frac{1}{p^{n-p+i}}}\right)^{p^{n-2}}\Lambda_{i-1,n}^{p^{n-1}-p^{n-2}}                                      \\
                      = & p \left((-1)^{n-2}(-1)^n\zetax p^{\frac{1}{p-1}-\frac{1}{p^{2-p+i}}}\right)\left(1+O\left(p^{\frac{1}{p^{n-1}(p-1)}}\right)\right)^{p^{n-2}(p-1)} +O\left(p^2\right)             \\
                      = & \zetax p^{1+\frac{1}{p-1}-\frac{1}{p^{2-p+i}}}\left(1+\sum_{r=1}^{p-1}\binom{p-1}{r}O\left(p^{\frac{r}{p^{n-1}(p-1)}}\right)\right)^{p^{n-2}}+O\left(p^2\right)      \\
                      = & \zetax p^{1+\frac{1}{p-1}-\frac{1}{p^{2-p+i}}}\left(1+O\left(p^{\frac{1}{p^{n-1}(p-1)}}\right)\right)^{p^{n-2}}+O\left(p^2\right) .
                  \end{split}
              \end{equation}
              Notice that
              \begin{align*}
                  \left(1+O\left(p^{\frac{1}{p^{n-1}(p-1)}}\right)\right)^{p^{n-2}}= & 1+\sum_{r=1}^{p^{n-2}}\binom{p^{n-2}}{r}O\left(p^{\frac{r}{p^{n-1}(p-1)}}\right) \\
                  =                                                                  & 1+\sum_{r=1}^{p^{n-2}}O\left(p^{n-2-v_p(r)+\frac{r}{p^{n-1}(p-1)}}\right)        \\
                  =                                                                  & 1+O\left(p^{\frac{1}{p(p-1)}}\right) .
              \end{align*}

              Since $1+\frac{1}{p-1}-\frac{1}{p^{2-p+i}}+\frac{1}{p(p-1)}>1+\frac{1}{p-1}$ for all $i\geq p$,     we can rewrite \Cref{eq:1060} as
              \begin{align*}
                    & \zetax p^{1+\frac{1}{p-1}-\frac{1}{p^{2-p+i}}}\left(1+O\left(p^{\frac{1}{p(p-1)}}\right)\right)+O\left(p^2\right)
                  \\
                  = & \zetax p^{1+\frac{1}{p-1}-\frac{1}{p^{2-p+i}}} +O\left(p^{1+\frac{1}{p-1}}\right) .
              \end{align*}

              \par Thus, by assembling the terms of valuation $\geq 1+\frac{1}{p-1}$ in \Cref{luckydog}, we obtain
              \begin{align*}
                  \Lambda_{i,n}^{p^{n-1}}-1= & \Lambda_{i-1,n}^{p^{n-1}}-1+\left((-1)^n\zetax p^{\frac{1}{p^{n-2}(p-1)}-\frac{1}{p^{n-p+i}}}\right)^{p^{n-1}} \\
                                             & +\zetax p^{1+\frac{1}{p-1}-\frac{1}{p^{2-p+i}}} +O\left(p^{1+\frac{1}{p-1}}\right)                             \\
                  =                          & \Lambda_{i-1,n}^{p^{n-1}}-1-\zetax p^{1+\frac{1}{p-1}-\frac{1}{p^{1-p+i}}}                                     \\
                                             & +\zetax p^{1+\frac{1}{p-1}-\frac{1}{p^{2-p+i}}} +O\left(p^{1+\frac{1}{p-1}}\right) .
              \end{align*}
              Finally, combining with the induction hypothesis, we obtain
              \begin{align*}
                  \Lambda_{i,n}^{p^{n-1}}-1=
                    & \sum_{l=1}^{p-1} \frac{(-1)^l}{[l!]}\zetax^l p^{\frac{l}{p-1}}+\zetax p^{1+\frac{1}{p-1}-\frac{1}{p^{i-p+1}}}+O\left(p^{1+\frac{1}{p-1}}\right)  \\
                    & -\zetax p^{1+\frac{1}{p-1}-\frac{1}{p^{1-p+i}}}+\zetax p^{1+\frac{1}{p-1}-\frac{1}{p^{2-p+i}}} +O\left(p^{1+\frac{1}{p-1}}\right)                \\
                  = & \sum_{l=1}^{p-1} \frac{(-1)^l}{[l!]}\zetax^l p^{\frac{l}{p-1}}+\zetax p^{1+\frac{1}{p-1}-\frac{1}{p^{2-p+i}}}+O\left(p^{1+\frac{1}{p-1}}\right).
              \end{align*}

    \end{enumerate}
\end{proof}

\begin{proposition}\label{seconde}  Let $n\in \bbN_{\geq 2}$.
    \begin{enumerate}
        \item
              For $1\leq i< p-1$, we have
              $$\Lambda_{i,n}^{p^n}-1=\frac{(-1)^i}{(i+1)!}\zetax^{i+1}p^{1+\frac{i+1}{p-1}}+o\left(p^{1+\frac{i+1}{p-1}}\right) .$$
        \item \label{lem:p_squre_power_expansion2}
              For $i\geq p-1$, we have
              $$\Lambda_{i,n}^{p^{n}}-1=\zetax p^{2+\frac{1}{p-1}-\frac{1}{p^{i-p+2}}}+O\left(p^{2+\frac{1}{p-1}}\right).$$
    \end{enumerate}
\end{proposition}
\begin{proof} \begin{enumerate}
        \item Recall that by \Cref{premier}, for  $1\leq i< p-1$, we have
              $$\Lambda_{i,n}^{p^{n-1}}=\sum_{l=0}^{i}\frac{(-1)^l}{[l!]}\zetax^l p^{\frac{l}{p-1}}+O\left(p^{1+\frac{1}{p(p-1)}}\right).$$
              Let $\tilde{\Lambda}_{i,n}= \sum_{l=0}^{i}\frac{(-1)^l}{l!}\zetax^l p^{\frac{l}{p-1}}=\sum_{l=0}^i\frac{\theta_n^l}{l!}$, with $\theta_n=-\zetax p^{\frac{1}{p-1}}$. By \labelcref{lem:stirling_gen:2} of \Cref{lem:stirling_gen}, for $1\leq k\leq p$, we have
              $$\left(\tilde{\Lambda}_{i,n}-1\right)^k=\left(\sum_{l=1}^i\frac{\theta_n^l}{l!}\right)^k=\sum_{l=k}^{ik}\frac{k!}{l!}\stirling{l}{k}_{\leq i}\theta_n^l$$
              We remark that $v_p(\theta_n)=\frac{1}{p-1}$, $v_p(\tilde{\Lambda}_{i,n})=0$ and
              $$ \tilde{\Lambda}_{i,n}-\Lambda_{i,n}^{p^{n-1}}= \sum_{l=0}^{i}(-1)^l\left(\frac{1}{l!}-\frac{1}{[l!]}\right)\zetax^l p^{\frac{l}{p-1}}+O(p^{1+\frac{1}{p(p-1)}}). $$
              For all $0\leq l\leq i< p-1$, we have $v_p(l!-[l!])\geq 1$; thus we have
              $$ \tilde{\Lambda}_{i,n}-\Lambda_{i,n}^{p^{n-1}}=\sum_{l=1}^{i}O\left( p^{1+\frac{l}{p-1}}\right)+O\left(p^{1+\frac{1}{p(p-1)}}\right),$$
              and we can rewrite $\Lambda_{i,n}^{p^n}-1$ as following:
              $$\Lambda_{i,n}^{p^n}-1=\left(\Lambda_{i,n}^{p^{n-1}}\right)^p-1= \left(\tilde{\Lambda}_{i,n}+O\left(p^{1+\frac{1}{p(p-1)}}\right) \right)^p-1=\tilde{\Lambda}_{i,n}^p-1+O\left(p^{2+\frac{1}{p(p-1)}}\right).$$
              We reduce to estimate the $p$-adic valuation of $\tilde{\Lambda}_{i,n}^p-1$. On the other hand, we have
              \begin{equation}\label{term}\tilde{\Lambda}_{i,n}^p-1=\sum_{k=1}^{p}\binom{p}{k}(\tilde{\Lambda}_{i,n}-1)^k=\sum_{k=1}^{p}\binom{p}{k}\sum_{l=k}^{ik}\frac{k!}{l!}\stirling{l}{k}_{\leq i}\theta^l_n.\end{equation}

              By \Cref{vrai},  we have $v_p\left(\frac{k!}{l!}\stirling{l}{k}_{\leq i}\right)\geq 0$ for any $k,l\in \bbN$. Thus, we can rewrite \Cref{term} by assembling the terms with valuation $> 1+\frac{i+1}{p-1}$:
              \begin{equation}
                  \begin{split}
                      \tilde{\Lambda}_{i,n}^{p}-1
                      =&o(p^{1+\frac{i+1}{p-1}})+\sum_{l=1}^{i+1}\frac{\theta_n^l}{l!}\sum_{k=1}^{l}\binom{p}{k}k!\stirling{l}{k}_{\leq i}+ \sum_{l=p}^{p+i}\frac{p!}{l!} \stirling{l}{k}_{\leq i}\theta_n^{l}
                      \\
                      =& o\left(p^{1+\frac{i+1}{p-1}}\right)+p\sum_{l=1}^{i+1}\frac{\theta_n^l}{l!}\sum_{k=1}^l(-1)^{k-1}(k-1)!\stirling{l}{k}_{\leq i}-p\sum_{l=1}^{i+1}\frac{p!}{(l+p-1)!}\stirling{l+p-1}{p}_{\leq i}\theta_n^{l} \\  =   & o\left(p^{1+\frac{i+1}{p-1}}\right)+p\sum_{l=1}^{i+1}\frac{\theta^l_n}{l!}\left(\sum_{k=1}^l(-1)^{k-1}(k-1)!\stirling{l}{k}_{\leq i}+\frac{l!p}{(l+p-1)!}\stirling{l+p-1}{p}_{\leq i}\right),
                  \end{split}
              \end{equation}
              where the last equality follows from $- (p-1)!\equiv 1\mod p$. Let $$G_i(l)=\left(\sum_{k=1}^l(-1)^{k-1}(k-1)!\stirling{l}{k}_{\leq i}\right)+\frac{p \cdot l!}{(l+p-1)!}\stirling{l+p-1}{p}_{\leq i}.$$
              Together with \Cref{preliminaire2}, we have
              \begin{align*}
                  \widetilde{\Lambda}_{i,n}^p-1= & o\left(p^{1+\frac{i+1}{p-1}}\right)+p\sum_{l=1}^{i+1}G_i(l)\frac{\theta_n^l}{l!}                                                                                                                                            \\
                  =                              & o\left(p^{1+\frac{i+1}{p-1}}\right)+p\left(\sum_{l=1}^i O(p)\frac{\theta_n^l}{l!}+(-1+O(p))\frac{\theta_n^{i+1}}{(i+1)!}\right)                                                                                             \\
                  =                              & o\left(p^{1+\frac{i+1}{p-1}}\right)+p\left(o(p)-\frac{\theta_n^{i+1}}{(i+1)!}+O\left(p^{1+\frac{i+1}{p-1}}\right)\right)                                                                                                    \\
                  =                              & o\left(p^{1+\frac{i+1}{p-1}}\right)-p\frac{\theta_n^{i+1}}{(i+1)!}                                                            =\frac{(-1)^i}{(i+1)!}\zetax^{i+1}p^{1+\frac{i+1}{p-1}}+o\left(p^{1+\frac{i+1}{p-1}}\right) .
              \end{align*}
              As a consequence, we have
              $$\Lambda_{i,n}^{p^{n}}-1=\frac{(-1)^i}{(i+1)!}\zetax^{i+1}p^{1+\frac{i+1}{p-1}}+o\left(p^{1+\frac{i+1}{p-1}}\right)  .$$
        \item
              Now suppose $i\geq p-1$.
              \par Let $\tilde{\Lambda}_{p-1,n}= \sum_{l=0}^{p-1} \frac{(-1)^l}{l!}\zetax^l p^{\frac{l}{p-1}}= \sum_{l=0}^{p-1}\frac{\theta_n^l}{l!}$, with $\theta_n=-\zetax p^{\frac{1}{p-1}}$.
              By \Cref{premier}, we have
              \begin{align*}
                    & \Lambda_{i,n}^{p^{n-1}}-\tilde{\Lambda}_{p-1,n}                                           \\
                  = & \sum_{l=0}^{p-1} (-1)^l\zetax^l\left(\frac{1}{[l!]}-\frac{1}{l!}\right) p^{\frac{l}{p-1}} \\
                    & +\zetax p^{1+\frac{1}{p-1}-\frac{1}{p^{i-p+2}}}+O\left(p^{1+\frac{1}{p-1}}\right)         \\
                  = & \zetax p^{1+\frac{1}{p-1}-\frac{1}{p^{i-p+2}}}+O\left(p^{1+\frac{1}{p-1}}\right) .
              \end{align*}

              Therefore, we have
              \begin{equation}\label{termp}
                  \begin{split}
                      \Lambda_{i,n}^{p^{n}}-1  = & \left(\Lambda_{i,n}^{p^{n-1}}\right)^p-1
                      =                          \left(\tilde{\Lambda}_{p-1,n}+\zetax p^{1+\frac{1}{p-1}-\frac{1}{p^{i-p+2}}}+O\left(p^{1+\frac{1}{p-1}}\right)\right)^p-1                                                          \\
                      =                          & \tilde{\Lambda}_{p-1,n}^p-1+\sum_{k=1}^p\binom{p}{k}\tilde{\Lambda}_{p-1,n}^{p-k}\left(\zetax p^{1+\frac{1}{p-1}-\frac{1}{p^{i-p+2}}}+O\left(p^{1+\frac{1}{p-1}}\right)\right)^k \\
                      =                          & \tilde{\Lambda}_{p-1}^p-1+\tilde{\Lambda}_{p-1,n}^{p-1}\left(\zetax p^{2+\frac{1}{p-1}-\frac{1}{p^{i-p+2}}}+O\left(p^{2+\frac{1}{p-1}}\right)\right)                             \\
                      & +\sum_{k=2}^p\binom{p}{k}\tilde{\Lambda}_{p-1,n}^{p-k}\left(\zetax p^{1+\frac{1}{p-1}-\frac{1}{p^{i-p+2}}}+O\left(p^{1+\frac{1}{p-1}}\right)\right)^k                            \\
                      =                          & \tilde{\Lambda}_{p-1,n}^p-1+\tilde{\Lambda}_{p-1,n}^{p-1}\zetax p^{2+\frac{1}{p-1}-\frac{1}{p^{i-p+2}}}+O\left(p^{2+\frac{1}{p-1}}\right).\end{split}
              \end{equation}

              Since $\tilde{\Lambda}_{p-1,n}^{p-1}=1+O\left(p^{\frac{1}{p-1}}\right)$, we may simplify \Cref{termp} as
              $$\Lambda_{i,n}^{p^{n}}-1  = \tilde{\Lambda}_{p-1,n}^p-1+\zetax p^{2+\frac{1}{p-1}-\frac{1}{p^{i-p+2}}}+O\left(p^{2+\frac{1}{p-1}}\right).$$
              By \Cref{boringdog} we have $\tilde{\Lambda}_{p-1,n}^p-1=\left(\sum_{l=0}^{p-1}\frac{\theta_n^l}{l!}\right)^p-1=O\left(p^{2+\frac{1}{p-1}}\right)$, therefore
              $$\Lambda_{i,n}^{p^{n}}-1=\zetax p^{2+\frac{1}{p-1}-\frac{1}{p^{i-p+2}}}+O\left(p^{2+\frac{1}{p-1}}\right),$$
              as expected.
    \end{enumerate}
\end{proof}

\subsection{Uniforminzer of $K_{2,n}$}\label{subsec:42109}
In this section, we use the expansion of $\zeta_{p^2}$ to get a uniformizer of $K_{2,n}$:
\begin{theorem}\label{thm:uniformizer}
    \begin{enumerate}
        \item The element
              $$\pi_{2,1}= \left(p^{\frac{1}{p}}\right)^{-1}\left(\zeta_{p^2}-\sum_{k=0}^{p-1}\frac{1}{[k!]}\zetax^kp^{\frac{k}{p(p-1)}}\right)$$
              is a uniformizer of $K_{2,1}$.
        \item For $m\geq 2$, the element
              $$\pi_{2,m}= \left(p^{\frac{1}{p^m}}\right)^{-\frac{p^m-1}{p-1}}\left(\zeta_{p^2}-\sum_{k=0}^{p-1}\frac{1}{[k!]}\zetax^kp^{\frac{k}{p(p-1)}}-\sum_{l=2}^m \zetax p^{\frac{1}{p-1}-\frac{1}{p^l}}\right)$$
              is a uniformizer of $K_{2,m}$.
    \end{enumerate}
\end{theorem}
\begin{proof}
    By \Cref{maintheorem1}, we know that
    $$\zeta_{p^2}=\sum_{k=0}^{p-1}\frac{1}{[k!]}\zetax^kp^{\frac{k}{p(p-1)}}+\sum_{k=2}^\infty\zetax p^{\frac{1}{p-1}-\frac{1}{p^k}}+O\left(p^{\frac{1}{p-1}}\right) .$$
    Therefore
    $$v_p(\pi_{2,1})= v_p\left(\sum_{k=2}^\infty\zetax p^{\frac{1}{p-1}-\frac{1}{p^k}}+O\left(p^{\frac{1}{p-1}}\right)\right)-\frac{1}{p}=\frac{1}{p-1}-\frac{1}{p^2}-\frac{1}{p}=\frac{1}{p^2(p-1)}=e_{K_{2,1}/\bbQ_p}^{-1}$$
    and similarly $v_p(\pi_{2,m})=\frac{1}{p^{m+1}(p-1)}=e_{K_{2,m}/\bbQ_p}^{-1}$ for $m\geq 2$.
    \par To see that $\pi_{2,1}\in K_{2,1}$, we can write $\frac{1}{[k!]}\zetax^kp^{\frac{k}{p(p-1)}}$ as $\left(\frac{1}{[k!]}p^{-\frac{k}{p}}\right)\left(\zetax p^{\frac{1}{p-1}}\right)^k$ and consequently $\pi_{2,1}\in\bbQ_p\left(\zeta_{p^2},\zetax p^{\frac{1}{p-1}}\right)$. By \Cref{lem:tame_iso}, this field is exactly $K_{2,1}$. Similarly, we have $\pi_{2,m}\in K_{2,m}$ for all $m\geq 2$, which finishes the proof.
\end{proof}

When doing the calculation, one should always make sure that the choice of $\zeta_{p^2}$ and $\zetax$ are compatible, i.e. $\zeta_{p^2}=1+\zetax p^{\frac{1}{p(p-1)}}+o\left(p^{\frac{1}{p(p-1)}}\right)$. To get over this inconvenience, one can replace $\zetax p^{\frac{1}{p-1}}$ with $1-\zeta_p$ and use the fact that $\zeta_p=1-\zetax p^{\frac{1}{p-1}}+O\left(p^{\frac{2}{p-1}}\right)$ to eliminate the appearance of $\zetax p^{\frac{1}{p(p-1)}}$ and $\zetax p^{\frac{1}{p-1}}$ in $\pi_{2,m}$:
\begin{corollary}\label{coro:32280}
    \begin{enumerate}
        \item The element
              $$\tilde{\pi}_{2,1}= \left(p^{\frac{1}{p}}\right)^{-1}\left(\zeta_{p^2}-\sum_{k=0}^{p-2}\frac{1}{[k!]}\left(\frac{1-\zeta_p}{p^{\frac{1}{p}}}\right)^k-p^{\frac{1}{p}}\right)$$
              is a uniformizer of $K_{2,1}$.
        \item For $m\geq 2$, the element
              $$\tilde{\pi}_{2,m}= \left(p^{\frac{1}{p^m}}\right)^{-\frac{p^m-1}{p-1}}\left(\zeta_{p^2}-\sum_{k=0}^{p-2}\frac{1}{[k!]}\left(\frac{1-\zeta_p}{p^{\frac{1}{p}}}\right)^k-p^{\frac{1}{p}}-(1-\zeta_p)\sum_{l=2}^m p^{-\frac{1}{p^l}}\right)$$
              is a uniformizer of $K_{2,m}$.
    \end{enumerate}
\end{corollary}
Another method proposed by Lampert without proof (cf. \cite{LampertMOF1}) to construct a uniformizer of $K_{2,2}$ (which can be generalized to arbitrary $K_{2,m}$ easily) is to consider the following sequence\footnote{We modifiy Lampert's original idea slightly to correct and simplify the result.}:
\begin{equation*}
    \left\{
    \begin{aligned}
        z_1:=     & \zeta_{p^2}-1-p^{\frac{1}{p}},                                                                                \\
        z_2:=     & z_1^{p-1}+p^{\frac{1}{p}}-p^{\frac{2p-1}{p^2}} ,                                                              \\
        z_{n+1}:= & z_n^{p-1}-\left(\left[C_{v_p(z_n)}\left(z_n\right)\right]p^{v_p(z_n)}\right)^{p-1}\text{, for }n=2,3,\cdots .
    \end{aligned}\right.
\end{equation*}
Then we can prove by keeping track of $\Supp{z_n}$ that:
\begin{proposition}\label{prop:9814}
    \begin{enumerate}
        \item There exists an integer $N\leq p$ such that $p^3(p-1)v_p(z_N)$ is an integer satisfying
              $$p^3(p-1)v_p(z_N)\equiv -p+1\pmod{p^2}.$$
        \item Let $M=p(p-1)(p-2)\sum_{i=1}^{N-1}v_p(z_n)+p$. Then $M\in\bbZ$ and for any solution $(a,b,c)\in\bbZ^3$ of the linear equation
              $$\left(p^2 M+1-p\right)a+p(p-1)b+p^2c=1,$$
              the element $z_N^a\cdot p^{b/p^2}\cdot \left(\zeta_{p^2}-1\right)^c$ is a uniformizer of $K_{2,2}$. In particular, one may take $(a,b,c)=(p+1,-p M,-2M+1)$ and
              $$\pi_{2,2}^\prime=\frac{z_N^{p+1}}{p^{M/p}\left(\zeta_{p^2}-1\right)^{2M-1}}$$
              is a uniformizer of $K_{2,2}$.
    \end{enumerate}

\end{proposition}

\begin{example}
    When $p=7$, our method (cf. \Cref{thm:uniformizer} and \Cref{coro:32280}) gives two uniformizers of $K_{2,2}$:
    $$\pi_{2,2}=7^{-\frac{8}{49}}\left(\zeta_{49}-1-\zeta_{12}7^{\frac{1}{42}}-\frac{1}{[2]}\zeta_6 7^{\frac{1}{21}}+\zeta_4 7^{\frac{1}{14}}-\frac{1}{[3]}\zeta_3 7^{\frac{2}{21}}-\zeta_{12}^5 7^{\frac{5}{42}}-7^{\frac{1}{7}}-\zeta_{12}7^{\frac{43}{294}}\right),$$
    \begin{gather*}
        \tilde{\pi}_{2,2}=7^{-\frac{8}{49}}\left(\zeta_{49}-1-\left(\frac{1-\zeta_7}{7^{\frac{1}{7}}}\right)-\frac{1}{[2]}\left(\frac{1-\zeta_7}{7^{\frac{1}{7}}}\right)^2+\left(\frac{1-\zeta_7}{7^{\frac{1}{7}}}\right)^3-\frac{1}{[3]}\left(\frac{1-\zeta_7}{7^{\frac{1}{7}}}\right)^4-\left(\frac{1-\zeta_7}{7^{\frac{1}{7}}}\right)^5-7^{\frac{1}{7}}-(1-\zeta_7)7^{\frac{1}{49}}\right),
    \end{gather*}
    while Lampert's method (cf. \Cref{prop:9814}) provides a more complicated uniformizer of the same field:
    \begin{gather*}
        \pi_{2,2}^\prime=\frac{\left(\left(\left(\left(\left(\left(\left(\zeta_{49}-1-7^{1/7}\right)^6+7^{1/7}-7^{13/49}\right)^6+7\right)^6+7^{43/7}\right)^6+7^{37}\right)^6+7^{1555/7}\right)^6+7^{1333}\right)^8}{7^{55987/7}\left(\zeta_{49}-1\right)^{111973}}.
    \end{gather*}

\end{example}

\bibliographystyle{alphaurl}
\bibliography{ref}
\end{document}